\documentclass[12pt, reqno]{amsart}
\usepackage{amsmath, amstext, amsbsy, amssymb, amscd,mathrsfs}
\usepackage{amsmath}
\usepackage{amsxtra}
\usepackage{amscd}
\usepackage{amsthm}
\usepackage{amsfonts}
\usepackage{amssymb}
\usepackage{eucal}
\usepackage{color}

\input prepictex
\input pictex
\input postpictex

\newtheorem{theorem}{Theorem}[section]
\newtheorem{lemma}[theorem]{Lemma}
\newtheorem{proposition}[theorem]{Proposition}
\newtheorem{corollary}[theorem]{Corollary}
\theoremstyle{definition}
\newtheorem{definition}[theorem]{Definition}
\newtheorem{example}[theorem]{Example}
\newtheorem{remark}[theorem]{Remark}

\newtheorem{question}[theorem]{Question}

\numberwithin{equation}{section}

\newcommand{\g}{\mathfrak{g}}

\newcommand{\h}{\mathfrak{h}}
\newcommand{\p}{\mathfrak{p}}

\newcommand{\m}{\mathfrak{m}}

\newcommand{\rea}[2]{U_{#1}(#2)} 
\newcommand{\Hom}{\text{Hom}}
\newcommand{\pth}[1]{{#1}^{[p]}} 
\newcommand{\ra}{\rightarrow}
\newcommand{\sudim}{\underline{\text{dim}}\,}
\newcommand{\ev}[1]{{#1}_{\bar{0}}}
\newcommand{\od}[1]{{#1}_{\bar{1}}}
\newcommand{\Zp}{\mathcal{Z}_p}
\newcommand{\osp}{\mathfrak{osp}}

\newcommand{\hf}{\frac12}
\newcommand{\la}{\lambda}

\newcommand{\ad}{\text{ad}\,}
\newcommand{\C}{ \mathbb C }

\newcommand{\gl}{{\mathfrak{gl}}}

\newcommand{\Z}{ \mathbb Z }

{\vskip-\lastskip\medskip
  \noindent
  {\em #1.}\enspace
  }%
{\qed\par\medskip
  }

\makeatletter
\def\Ddots{\mathinner{\mkern1mu\raise\p@
\vbox{\kern7\p@\hbox{.}}\mkern2mu
\raise4\p@\hbox{.}\mkern2mu\raise7\p@\hbox{.}\mkern1mu}}
\makeatother

\begin{document}
\title[Lie Superalgebras in Prime Characteristic]
{Representations of Lie Superalgebras in Prime Characteristic I}

\author[Weiqiang Wang]{Weiqiang Wang}\thanks{
This research is partially supported by NSA and NSF grants.}
\address{Department of Mathematics, University of Virginia,
Charlottesville, VA 22904} \email{ww9c@virginia.edu (Wang)\\
lz4u@virginia.edu (Zhao)}

\author{Lei Zhao}

\begin{abstract}
We initiate the representation theory of restricted Lie
superalgebras over an algebraically closed field of characteristic
$p>2$. A superalgebra generalization of the celebrated
Kac-Weisfeiler Conjecture is formulated, which exhibits a mixture
of $p$-power and $2$-power divisibilities of dimensions of
modules. We establish the Conjecture for basic classical Lie
superalgebras.
\end{abstract}

\maketitle
\date{}
  \setcounter{tocdepth}{1}
  \tableofcontents

\section{Introduction}
\subsection{ }

The finite-dimensional complex simple Lie superalgebras were
classified by Kac \cite{Kac} in 1970's. A main subclass in this
classification, independently obtained by Scheunert, Nahm and
Rittenberg \cite{SNR1,SNR2} (with contributions from Kaplansky and
others), is called {\em basic classical} Lie superalgebras, which
by definition admit an even nondegenerate supersymmetric bilinear
form and whose even subalgebras are reductive. It consists of
several infinite series and 3 exceptional ones.

The modular representations of restricted Lie algebras in prime
characteristic have been developed over the years with intimate
connections to algebraic groups (see \cite{KW, FP, Pr1}; cf.
Jantzen \cite{Jan1} for a review). In \cite{Pr1} Premet developed
new ideas to establish a long-standing conjecture of Kac and
Weisfeiler \cite{KW} for Lie algebras of reductive groups. Among
other things Premet made a crucial use of the powerful machinery
of support varieties for Lie algebras developed by
Friedlander-Parshall \cite{FP1}, Jantzen and others.

\subsection{ }

In this paper and its sequels we will initiate and develop
systematically the modular representations of Lie superalgebras
over an algebraically closed field $K$ of characteristic $p>2$. To
our best knowledge, there has not been any serious study in this
direction, perhaps because the representation theory of simple Lie
superalgebras over $\C$ (e.g. the irreducible character problem)
is already very difficult and remains to be better understood. It
turns out that the modular super representation theory is a very
promising direction full of nontrivial yet accessible problems and
conjectures, and it exhibits a novel phenomenon of nondefining
characteristics with primes $p$ and $2$.

There has been increasing interest in modular representation
theory of algebraic supergroups in connection with other areas in
recent years, thanks to the work of Brundan, Kleshchev, Kujawa and
others (see \cite{SW} for references and historical remarks). As
usual, when $\g$ is the Lie (super)algebra of an algebraic
supergroup $G$, the study of $\g$-modules with zero $p$-character
corresponds to the study of rational modules of $G$ or of its
Frobenius kernel. The current work puts the study of restricted
modules of Lie superalgebras in a wider context. It is possible
that some basic duality between categories of modular
representations of Lie superalgebras and Lie algebras will emerge
when the super theory is more adequately developed.

\subsection{}

For a (finite-dimensional) restricted Lie superalgebra $\g
=\ev{\g} \oplus \od{\g}$, one defines a $p$-character $\chi \in\ev
\g^*$ and the associated reduced enveloping superalgebra $U_\chi
(\g)$.
Let $\g_{\chi} =\g_{\chi,\bar{0}} \oplus \g_{\chi,\bar{1}}$ be the
centralizer of $\chi$ in $\g$ of codimension $d_0|d_1$. It is well
known that $d_0$ is even, but $d_1$ can possibly be odd. Let
$\lfloor a \rfloor$ denote the least integer upper bound of $a$.
We formulate at the end of  Section~\ref{sec:basics} the {\em
Super KW Conjecture} or {\em Super KW Property} which asserts that
{\em every $\rea{\chi}{\g}$-module has dimension divisible by
$p^{\frac{d_0}{2}}2^{\lfloor \frac{d_1}{2} \rfloor}$.} The
celebrated Kac-Weisfeiler conjecture (Premet's theorem) states
that the KW Property above holds for the Lie algebra $\g$ of a
reductive algebraic group (with $\od \g =0$ and $d_1=0$ above)
under mild assumptions on $p$.

The main result of this paper is that the Super KW Property as
formulated above holds for basic classical Lie superalgebras (the
general linear Lie superalgebras, though not simple, are also
included). In this paper we shall exclude the usual simple Lie
algebras from the basic classical Lie superalgebras, even though
all the proofs make sense for them as well.

For restricted Lie superalgebras, the theory of support varieties
has yet to be developed (see however an interesting construction
over $\C$ of Boe, Kujawa and Nakano \cite{BKN}). Our approach will
take full advantage of a combination of techniques developed in
Premet \cite{Pr1, Pr2} and in Skryabin \cite{Skr}, which allow us
to establish the Super KW Property for basic Lie superalgebras
with nilpotent $p$-characters bypassing completely the support
variety machinery. In addition, we establish a Morita equivalence
to reduce general $p$-characters to nilpotent ones, adapting
Friedlander-Parshall \cite{FP} (also see \cite{KW}) to the
superalgebra setup. At several places we have to find ways to
overcome new implications and difficulties which are not presented
in the usual Lie algebra setup. Let us explain in some detail.

\subsection{}

In Section~\ref{sec:grading} we construct a natural $\Z$-grading
on an arbitrary basic classical Lie superalgebra $\g$ associated
to a given nilpotent $p$-character $\chi$. For the Lie
superalgebras of type $\mathfrak{sl}$, $\mathfrak{gl}$ and
$\mathfrak{osp}$, our explicit construction, which works for every
prime $p>2$, is built on the one in Jantzen \cite{Jan2} for the
classical Lie algebras. For the exceptional Lie superalgebras, we
impose somewhat stronger conditions (which can presumably be
relaxed) on the prime $p$.

In Section~\ref{sec:proof}, we construct a $p$-subalgebra
$\mathfrak{m}$ of $\g$ from such a $\Z$-grading associated to a
nilpotent $p$-character $\chi$, following Premet \cite{Pr1} (the
cases with $d_1$ odd offer some new perspectives). We then use the
ingenious and elementary method in Skryabin \cite{Skr} to prove
that every simple $U_\chi (\g)$-module is free over the algebra
$U_\chi(\mathfrak{m})$ of dimension $\delta:
=p^{\frac{d_0}{2}}2^{\lceil \frac{d_1}{2} \rceil}$, where $\lceil
a \rceil$ denotes the largest integer lower bound of $a$. We are
done if $d_1$ is even; for $d_1$ odd, $\lfloor \frac{d_1}{2}
\rfloor -\lceil \frac{d_1}{2} \rceil=1$, and an extra factor $2$
required in the Super KW Conjecture is then supplied by a
$2$-dimensional endomorphism algebra of simple modules ``of type
$Q$", which is a pure super phenomenon.

For a basic classical Lie superalgebra $\g$, we further construct
a $K$-superalgebra (called a finite $W$-superalgebra)
$$
W_\chi(\g) = \text{End}_{\rea{\chi}{\g}}(
\rea{\chi}{\g}\otimes_{\rea{\chi}{\m}}K_{\chi}).
$$
and show that $\rea{\chi}{\g}$ is isomorphic to the matrix algebra
$M_{\delta} (W_\chi(\g)^{op})$, following Premet's argument
\cite{Pr2} with a modification to avoid completely the use of
support variety machinery. Again the type $Q$ phenomenon is
implicit behind the scene here. This provides a conceptual
explanation for the Super KW Property of $\g$. The structures and
representations of the superalgebra $W_\chi(\g)$ and its complex
counterpart deserve to be studied separately.

One can verify by inspection that the centralizer of a semisimple
(even) element in $\g$ is always a Levi subalgebra of $\g$, and
the case-by-case verification is elementary yet tedious due to the
existence of non-conjugate Borel subalgebras. In
Section~\ref{sec:gen. char.}, we establish a Morita equivalence
theorem which relates the reduced enveloping algebra
$\rea{\chi}{\g}$ with an arbitrary $p$-character $\chi$ of $\g$ to
that of a Levi subalgebra of $\g$ with a nilpotent $p$-character.
This is a superalgebra generalization of the classical results
\cite{KW, FP}, and our proof follows largely the general strategy
in \cite{FP}, with a few modifications to deal with the new super
features (e.g. the existence of non-conjugate Borel subalgebras
and three types of roots which give rise to three rank one Lie
superalgebras). This Morita equivalence together with the results
in Section~\ref{sec:proof} completes the proof of the Super KW
Conjecture for $\g$.

Finally in Section~\ref{sec:osp12}, we work out by hand completely
the representation theory of the simple Lie superalgebra
$\mathfrak{osp}(1|2)$ for all $p$-characters, which is very
similar to the $\mathfrak{sl}(2)$ case with some additional
interesting super type $Q$ phenomenon. In particular, we show that
there is no projective simple $U_\chi
(\mathfrak{osp}(1|2))$-module when $\chi$ is zero or regular
nilpotent, in contrast to the $\mathfrak{sl}(2)$ case.


\subsection{} Throughout we work with an algebraically
closed field $K$ with characteristic $p > 2$ as the ground field
(unless specified otherwise for the exceptional Lie
superalgebras). We exclude $p=2$ since in that case Lie
superalgebras coincide with Lie algebras.

A superspace is a $\Z_2$-graded vector space $V = \ev{V} \oplus
\od{V}$, in which we call elements in $\ev{V}$ and $\od{V}$ even and
odd, respectively . Write $|v| \in \Z_2$ for the parity (or degree)
of $v \in V$, which is implicitly assumed to be $\Z_2$-homogeneous.
A bilinear form $f$ on $V$ is {\em supersymmetric} if $f(u,v)
=(-1)^{|u||v|} f(v,u)$ for all homogeneous $u,v \in V$. We will use
the notation
$$
\sudim V = \dim \ev{V} | \dim \od V;\qquad \dim V =\dim \ev V +
\dim \od V.
$$
All Lie superalgebras $\g$ will be assumed to be finite
dimensional.  We will use $U(\g)$ to denote its universal
enveloping superalgebra.

According to Walls \cite{W}, the finite-dimensional simple
associative superalgebras over $K$ are classified into two types:
besides the usual matrix superalgebra (called type $M$) there are
in addition simple superalgebras of type $Q$. Alternatively, a
superalgebra analogue of Schur's Lemma states that the
endomorphism ring of an irreducible module of a superalgebra is
either one-dimensional or two-dimensional (in the latter case it
is isomorphic to a Clifford algebra), cf. e.g. Kleshchev
\cite[Chap.~12]{Kle}. An irreducible module is {\em of type $M$}
if its endomorphism ring is one-dimensional and it is {\em of type
$Q$} otherwise.
%

By vector spaces, derivations, subalgebras, ideals, modules, and
submodules etc. we mean in the super sense unless otherwise
specified.

\section{Basic results for restricted Lie superalgebras}  \label{sec:basics}

The materials in this Section (except
Subsect.~\ref{sec:KWproperty}) are standard generalizations from
Lie algebras and should not be surprising to experts, but we find
it convenient to formulate them precisely for the latter use.

\subsection{Restricted Lie Superalgebras}

The notion of restricted Lie superalgebras can be easily
formulated as follows (cf. e.g. Farnsteiner \cite{Far}).
\begin{definition}\label{def:res}
A Lie superalgebra $\g = \ev{\g} \oplus \od{\g}$ is called a \em
{restricted} Lie superalgebra, if there is a $p$th power map $\ev
\g \ra \ev \g$, denoted as $\pth{}$, satisfying
\begin{enumerate}
\item[(a)] $\pth{(kx)} = k^p \pth x$ for all  $k \in K$ and $x \in
\ev \g$,
\item[(b)] $[\pth x , y] = (\ad x)^p (y)$ for all $x \in \ev \g$
and $y \in \g$,
\item[(c)] $\pth{(x + y)} = \pth x + \pth y + \sum_{i= 1}^{p-1}
s_i(x,y)$ for all $x, y \in \ev \g$ where $is_i$ is the
coefficient of ${\lambda}^{i-1}$ in $(\ad(\lambda x +
y))^{p-1}(x)$.
\end{enumerate}
\end{definition}
In short, a restricted Lie superalgebra is a Lie superalgebra
whose even subalgebra is a restricted Lie algebra and the odd part
is a restricted module by the adjoint action of the even
subalgebra.

For example, the Lie algebra of an algebraic supergroup is
restricted (see \cite{SW}).

All the Lie (super)algebras in this paper will be assumed to be
restricted.

\subsection{Basic classical Lie superalgebras}\label{sec:bcLsa-p}

Following \cite{Kac} (and \cite{SNR1, SNR2}), we recall the list of
basic classical Lie superalgebras. In {\em loc. cit.}, the ground
field is $\C$. We observe that these Lie superalgebras (whose even
subalgebras are Lie algebras of reductive algebraic groups) are well
defined over $K$ and remain to be simple over $K$ of characteristic
$p>2$ (and $p>3$ in case of $D(2,1,\alpha)$ and $G(3)$), and they
admit an nondegenerate even supersymmetric bilinear form. This is
clear for the infinite series, and for the exceptional Lie
superalgebras it follows from the construction of \cite{SNR1, SNR2}
(it should be also possible via the method of contragredient Lie
superalgebras of Kac, cf. \cite{Kac} and the references therein).

In Table 1, we list all the basic classical Lie superalgebras over
$K$ with restrictions on $p$ (the general linear Lie superalgebra,
though not simple, is also included). We impose somewhat stronger
restrictions on the prime $p$ (which can probably be relaxed) for
the $3$ exceptional Lie superalgebras for our latter purpose of
constructing suitable $\Z$-gradings.

\vspace{.4cm}

\begin{center}
\begin{tabular}{|c|c|}
\hline
Lie superalgebra & Characteristic of $K$\\
\hline $\gl(m|n)$ & $p > 2$\\
\hline $\mathfrak{sl}(m|n)$ & $p > 2, p\nmid (m-n)$\\
\hline $B(m,n), C(n), D(m,n)$ & $p > 2$ \\
\hline $D(2,1;  \alpha)$ & $p > 3$ \\
\hline $F(4)$ & $p >15$ \\
\hline $G(3)$ & $p >15$ \\
\hline
\end{tabular}

\vspace{.1cm}
 TABLE 1: basic classical Lie $K$-superalgebras
\end{center}

\vspace{.4cm}

In this table, Lie superalgebra $D(2,1; \alpha)$, $\alpha \in K^*
\backslash \{0, -1\}$, is $17$-dimensional for which
$D(2,1;\alpha)_{\bar 0}$ is a Lie algebra of type $A_1 \oplus A_1
\oplus A_1$ and its adjoint representation on $D(2,1;\alpha)_{\bar
1}$ is $V \otimes V \otimes V$, where $V$ is the natural
representation of $A_1$.

The Lie superalgebra $F(4)$ is $40$-dimensional for which
$F(4)_{\bar 0}$ is a Lie algebra of type $B_3 \oplus A_1$ and its
adjoint representation on $F(4)_{\bar 1}$ is $U \otimes V$, where
$U$ is the $8$-dimension spin representation of $B_3$.

The Lie superalgebra $G(3)$ is $31$-dimensional for which
$G(3)_{\bar 0}$ is a Lie algebra of type $G_2 \oplus A_1$, and the
adjoint $G(3)_{\bar 0}$-module $G(3)_{\bar 1}$ is the tensor
product of the $7$-dimensional simple $G_2$-module with $V$. (We
thank Zongzhu Lin for helpful clarification on restriction of
prime $p$ for representations of Lie algebra of type $G_2$.)


\subsection{Reduced Enveloping Superalgebras}
Let $\g$ be a restricted Lie superalgebra. For each $x \in
\ev{\g}$, the element $x^p - \pth x \in U(\g)$ is central by
Definition \ref{def:res}. We refer to $\Zp(\g) = K\langle x^p -
\pth x \, |\,x \in \ev \g \rangle$ as the $p$-center of $U(\g)$.
Let $x_1, \ldots, x_s$ (resp. $y_1, \ldots, y_t$) be a basis of
$\ev \g$ (resp. $\od \g$). The following proposition is a
consequence of the PBW theorem for $U(\g)$.

\begin{proposition}\label{prop:pbw-U}
Let $\g$ be a restricted Lie superalgebra. Then $\Zp(\g)$ is a
polynomial algebra isomorphic to $K[x_i^p - x_i^{[p]} \vert \;
i=1,\ldots,s]$, and the enveloping superalgebra $U(\g)$ is free
over $\Zp(\g)$ with basis
\[
\{x_1^{a_1} \cdots x_s^{a_s} y_1^{b_1} \cdots y_t^{b_t} \,|\, 0
\leq a_i < p; \, b_j = 0, 1 \text{ for all }i, j \}.
\]
\end{proposition}


\begin{proposition}
 Let $\g$ be a restricted Lie superalgebra. Then all irreducible
$U(\g)$-modules are finite-dimensional, and their dimensions are
bounded by a constant $M(\g)$ which depends only on $\g$.
\end{proposition}

\begin{proof}
The proof (using Proposition~\ref{prop:pbw-U}) is the same as in
the non-super case (\cite{Cur}), and will be skipped.
\end{proof}

Let $V$ be a simple $U(\g)$-module and $x \in \ev \g$. By Schur's
lemma, the central element $x^p - \pth x$ acts by a scalar
$\zeta(x)$, which can be written as $\chi_V(x)^p$ for some $\chi_V
\in \ev \g^*$. We call $\chi_V$ the {\em $p$-character} of the
module $V$.

Fix $\chi \in \ev \g^*$. Let $I_{\chi}$ be the ideal of $U(\g)$
generated by the even central elements $x^p - \pth x - \chi(x)^p$.
The quotient algebra $\rea{\chi}{\g} := U(\g)/I_{\chi}$ is called
{\em the reduced enveloping superalgebra} with $p$-character
$\chi$. We often regard $\chi \in \g^*$ by letting $\chi (\od \g)
=0$. For $\chi = 0$, then $U_0(\g)$ is called the {\em restricted
enveloping superalgebra}.

The following proposition is an easy consequence of
Proposition~\ref{prop:pbw-U}.

\begin{proposition}\label{prop:pbw-u}
The superalgebra $\rea{\chi}{\g}$ has a basis
\[
\{x_1^{a_1} \cdots x_s^{a_s} y_1^{b_1} \cdots y_t^{b_t} \,|\, 0
\leq a_i < p; \, b_j = 0, 1 \, \text{for all }i, j \}.
\]
In particular, $\dim \rea{\chi}{\g} = p^{\dim \ev\g} 2^{\dim
\od\g}$.
\end{proposition}

\begin{remark}\label{rem:rea}
The adjoint algebraic group $\ev G$ of $\ev \g$ acts on $\g$ by
adjoint action since $\od \g$ is a rational $\ev G$-module. The
representation theory of the superalgebra $\rea{\chi}{\g}$ depends
only on the orbit of $\chi$ under the coadjoint action of $\ev G$,
since $\rea{\chi}{\g}\simeq \rea{\chi'}{\g}$ for $\chi$ and
$\chi'$ in the same orbit.
\end{remark}

\subsection{The baby Verma modules}\label{sec:ind-mod}
Let $\g$ be one of the Lie superalgebras in
Section~\ref{sec:bcLsa-p}. Fix a triangular decomposition
\[
\g =\mathfrak{n}^- \oplus \mathfrak{h} \oplus \mathfrak{n}^+,
\]
where $\mathfrak{n}^+ =\mathfrak{n}^+_{\bar 0}
+\mathfrak{n}^+_{\bar 1} $ (resp. $\mathfrak{n}^-$) is the Lie
subalgebra of positive (resp. negative) root vectors, and
$\mathfrak{h}$ is a Cartan (even) subalgebra of $\ev \g$ of rank
$l$.

By Remark~\ref{rem:rea} we may choose $\chi \in \ev \g^*$ with
$\chi(\mathfrak{n}^+_{\bar 0}) = 0$ without loss of generality.
Let $\mathfrak{b}= \mathfrak{h} \oplus \mathfrak{n}^+$.
Each $\lambda \in \mathfrak{h}^*$ defines a one-dimensional
$\mathfrak{h}$-module $K_{\lambda}$ where each $h \in
\mathfrak{h}$ acts as multiplication by $\lambda(h)$. The module
$K_{\lambda}$ is a $\rea{\chi}{\mathfrak{h}}$-module if and only
if $\lambda(h)^p - \lambda(\pth h) = \chi(h)^p$ for all $h \in
\mathfrak{h}$.
Set
\begin{align*}
\Lambda_{\chi} &= \{ \lambda \in \mathfrak{h}^* \vert \;
\lambda(h)^p - \lambda(\pth h) = \chi(h)^p \text{ for all } h \in
\mathfrak{h} \}.
\end{align*}
Note that $|\Lambda_{\chi}| = p^{\dim \h}$. Now for $\lambda \in
\Lambda_{\chi}$, $K_\la$ is regarded as a $\mathfrak{b}$-module
with $\mathfrak{n}^+$ acting trivially, and the {\em baby Verma
module} is defined to be the induced module
\[
Z_{\chi}(\lambda) : = \rea{\chi}{\g}
\otimes_{\rea{\chi}{\mathfrak{b}}} K_{\lambda}.
\]

\subsection{$p$-nilpotent Lie superalgebras}

A restricted Lie superalgebra $\g$ is called {\em $p$-nilpotent}
if, for each $x \in \ev \g$, there exists $r > 0$ such that
$x^{[p]^r} = 0$.

\begin{proposition} \label{!simple}
If $\g$ is a $p$-nilpotent Lie superalgebra, then each
$\rea{\chi}{\g}$ has a unique simple module (up to isomorphism).
Moreover, if $\chi =0$, then the trivial module is the only simple
module.
\end{proposition}

\begin{proof}
The proof is essentially the same as for $p$-nilpotent Lie
algebras (see, for example, proofs of Proposition 3.2 and
Theorem~3.3 in \cite{Jan1}), and will be skipped here.
\end{proof}

\subsection{$U_\chi (\g)$ as Frobenius and symmetric algebras}
Recall that the supertrace of an endomorphism $X$ on a vector
space $\ev V \oplus \od V$ is defined to be $ \text{str}(X) =
\text{tr}(X|_{\ev V}) - \text{tr}(X|_{\od V}).$ An associative
superalgebra $A$ with a supersymmetric nondegenerate bilinear form
will be referred to as a {\em symmetric (super)algebra}. The
standard properties for a symmetric algebra remain valid for
symmetric superalgebras. The following is a generalization of a
result of Friedlander-Parshall \cite{FP} for restricted Lie
algebras, and it can be proved in the same way.
\begin{proposition}\label{prop:Fro}
For each $\chi \in \ev \g^*$, the reduced enveloping superalgebra
$\rea{\chi}{\g}$ is a Frobenius algebra (with a nondegenerate
bilinear form denoted by $\bar{\mu}$). Moreover, $\bar{\mu}$ is
supersymmetric if and only if $\text{str}(\ad x) = 0$ for all $x
\in \ev \g$. In particular, if $\g$ is simple, then $\rea{\chi}{\g}$
is a symmetric superalgebra.
\end{proposition}

\subsection{The Super KW Property}  \label{sec:KWproperty}

Let $\chi\in \ev \g^*$ and we always regard $\chi \in \g^*$ by
setting $\chi(\od \g) = 0$. Denote the centralizer of $\chi$ in
$\g$ by $\g_{\chi} = \g_{\chi,\bar 0} + \g_{\chi,\bar 1}$, where
$\g_{\chi,i} = \{ y \in \g_i \vert \; \chi([y, \g]) = 0 \}$ for $i
\in \Z_2$. Set $d_i = \dim \g_i - \dim \g_{\chi,i}$. Recall that
$\lfloor a \rfloor$ denotes the least integer upper bound of $a$.

We formulate the following superalgebra generalization of the
Kac-Weisfeiler Conjecture.

\vspace{.2cm}

{\noindent \bf Super KW Property}. {\em The dimension of every
$\rea{\chi}{\g}$-module is divisible by
$p^{\frac{d_0}{2}}2^{\lfloor \frac{d_1}{2} \rfloor}$.}


It is well known that $d_0$ is an even integer and $d_1$ can
possibly be odd. When $\g$ is the Lie algebra of a reductive
algebraic group (i.e. $\od \g =0$), the celebrated Kac-Weisfeiler
(KW) conjecture \cite{KW} states that the (Super) KW Property
holds (where no $2$-power is involved). The original KW conjecture
has been completed by Premet \cite{Pr1}.

It is known that the KW Property holds for many interesting
examples beyond the setup of the original KW conjecture, and
nevertheless, it does not hold for all Lie algebras (and so
superalgebras).
\begin{question}
For which Lie superalgebras does the Super KW Property hold?
\end{question}

The main goal of this paper is to establish the Super KW
Conjecture for all the basic classical Lie superalgebras with
assumptions on $p$ as in the Table of Section~\ref{sec:bcLsa-p}.
{\bf In the remainder of this paper, $\g$ is assumed to be one of
the basic classical Lie superalgebras with such a restriction on
$p$}.

\section{The $\Z$-gradings of Lie superalgebras}\label{sec:grading}

\subsection{The $\Z$-gradings with favorable properties}
Let $\g$ be one of the basic classical Lie superalgebras in
Sect.~\ref{sec:bcLsa-p}. The Lie superalgebra $\g$ admits a
nondegenerate invariant even bilinear form $(\cdot, \cdot)$, whose
restriction on $\ev \g$ gives an isomorphism $\g_{\bar{0}}
\xrightarrow{\sim} \g_{\bar{0}}^*$. Let $\chi \in \g_{\bar{0}}^*$
be a nilpotent character, that is, it is the image of some
nilpotent element $X$ in $\g_{\bar{0}}$ under the above
isomorphism. Then $\g_{\chi}$ is equal to the usual centralizer
$\g_X = \{y \in \g |\; [X, y]= 0\}$.

\begin{theorem}\label{th:grading}
Let $\g$ be one of the basic classical Lie superalgebras in
\ref{sec:bcLsa-p}. Then there exists a $\Z$-grading $\g =
\oplus_{k \in \Z}\g(k)$ satisfying:
\begin{equation}\label{grading1}
X \in \g(2);
\end{equation}
\begin{equation}\label{grading2}
(\g(k), \g(l)) = 0, \quad \text{if }  k+l \neq 0;
\end{equation}
\begin{equation}\label{grading3}
\g_X = \oplus_{k \in \Z} \g_X(k) \qquad \text{where }
\g_X(k)=\g_X \cap \g(k);
\end{equation}
\begin{equation}\label{grading4}
\g_X(s)=0 \qquad \forall \, s < 0;
\end{equation}
\begin{equation}\label{grading5}
\sudim \g_X = \sudim \g(0) + \sudim \g(1).
\end{equation}
This grading is compatible with the $\Z_2$-grading, i.e.
$\g(k)=\g(k)_{\bar 0} \oplus \g(k)_{\bar 1}$ where $\g(k)_i =
\g(k) \cap \g_i$ for $i \in \Z_2, k \in \Z$.
\end{theorem}

We shall construct the $\Z$-gradings in the following subsections
according to the types of the Lie superalgebras. The constructions
for the Lie superalgebras of $\gl$ and $\osp$ types are natural
generalizations of those for classical Lie algebras, cf.
\cite[Chap.~1-5]{Jan2}.

\subsection{The $\Z$-gradings for $\gl$}\label{sec:gl-grading}

Let $V= V_{\bar{0}} \oplus V_{\bar{1}}$ be a vector space with
$\sudim V =m|n$. Identify $\gl(m| n)$ with $\gl(V)$, whose even
part is $\gl(V)_{\bar{0}}  =\gl(V_{\bar{0}}) \oplus
\gl(V_{\bar{1}})$. To a nilpotent element $X$ in
$\gl(V)_{\bar{0}}$, we associate with a pair of partitions
$(\pi_0, \pi_1)$, where $\pi_i = (\lambda_1^i, \ldots,
\lambda_{r_i}^i)$ of length $r_i$ is the shape of the Jordan
canonical form of the summand of $X$ in $\gl(V_i)$ respectively.

There exist $v_1, \ldots, v_{r_0} \in V_{\bar{0}}$ (resp. $u_1,
\ldots, u_{r_1} \in V_{\bar{1}}$) such that all $X^j v_i$ (resp.
$X^j u_i$) with $1 \leq i \leq r_0 \;(\text{resp. } r_1)$, $0 \leq
j < \lambda^0_i \; (\text{resp. } 0 \leq j < \lambda^1_i)$ are a
basis for $V_{\bar{0}}$ (resp. $V_{\bar{1}}$) and such that
$X^{\lambda^0_i}v_i = 0$ (resp. $X^{\lambda^1_i}u_i = 0$).

Each homogeneous $Z \in \gl(V)_X$ is determined by $Z(v_i)$ and
$Z(u_j)$ with $1 \leq i \leq r_0$, $1 \leq j \leq r_1$ because
$Z(X^k v_i)= X^k Z(v_i)$ and $Z(X^k u_j)= X^k Z(u_j)$ for all $i,
j$, and $k$. Furthermore, $X^{\lambda^0_i}Z(v_i)=0$ and
$X^{\lambda^1_j}Z(u_j)= 0$ for all $i$ and $j$. Using this, one
checks that if $Z$ is even, then $Z(v_i)$ and $Z(u_j)$ have the
forms
\begin{equation}\label{eq:cen-gl-ev1}
Z(v_i)= \sum_{l=1}^{r_0} \sum_{k=\max\{0, \lambda^0_l -
\lambda^0_i\}}^{\lambda^0_l - 1} a_{k,l;i}X^k v_l,
\end{equation}
\begin{equation}\label{eq:cen-gl-ev2}
Z(u_j)= \sum_{l=1}^{r_1} \sum_{k=\max\{0, \lambda^1_l -
\lambda^1_j\}}^{\lambda^1_l - 1} b_{k,l;j}X^k u_l.
\end{equation}
If $Z$ is odd, then
\begin{equation}\label{eq:cen-gl-od1}
Z(v_i)= \sum_{l=1}^{r_1} \sum_{k=\max\{0, \lambda^1_l -
\lambda^0_i\}}^{\lambda^1_l - 1} c_{k,l;i}X^k u_l,
\end{equation}
\begin{equation}\label{eq:cen-gl-od2}
Z(u_j)= \sum_{l=1}^{r_0} \sum_{k=\max\{0, \lambda^0_l -
\lambda^1_j\}}^{\lambda^0_l - 1} d_{k,l;j}X^k v_l.
\end{equation}
The coefficients $a_{k,l;i}$, $b_{k,l;j}$, $c_{k,l;i}$, and
$d_{k,l;j}$ above can be chosen arbitrarily in $K$. We compute
that
\begin{align*}
\dim \gl(V)_{X,{\bar 0}} &= \sum_{i,j=1}^{r_0} 
\text{min}(\lambda_i^0, \lambda_j^0) + \sum_{i,j=1}^{r_1}
\text{min}(\lambda_i^1, \lambda_j^1),\\
\dim \gl(V)_{X, {\bar 1}} &= 2 \sum_{i=1}^{r_0} \sum_{j=1}^{r_1}
\text{min}(\lambda^0_i, \lambda^1_j).
\end{align*}

Define a $\Z$-grading of the vector space $V$ (which is compatible
with the $\Z_2$-grading) as follows. Set $V(k)_{\bar{0}}$ (resp.
$V(k)_{\bar{1}}$) equal to the span of $X^j v_i$ with $k=2j + 1 -
\lambda^0_i$ (resp. $X^j u_i$ with $k= 2j + 1 - \lambda^1_i$).
Note that
\begin{equation}\label{eq:grading-gl-0}
\sudim V(l) = \sudim V(-l),
\end{equation}
\begin{equation}\label{eq:grading-gl-1}
X V(l) \subseteq V(l + 2).
\end{equation}

\begin{remark}
The grading defined above can be thought of coming from the
$\mathfrak{sl}(2)$-theory in characteristic zero. Indeed, let
$e,f$, and $h$ be the standard basis of $\mathfrak{sl}(2)$. We can
make $V$ into an $\mathfrak{sl}(2)$-module such that $e$ acts as
$X$, $h$ acts as on each $V(m)$ as multiplication with $m$, and
$f$ annihilates all $v_i$ and $u_j$, and maps each $X^k v_i$ and
$X^k u_j$ to suitable multiples of $X^{k-1}v_i$ and $X^{k-1}u_j$
respectively.
\end{remark}

The grading on $V$ induces a grading $\gl(V)= \oplus_{k \in \Z}
\gl(V)(k)$ with
\begin{equation}\label{glgrading1}
\gl(V)(k)=\{ f \in \gl(V) | f(V(l)) \subset V(l + k) \quad
\text{for all} \, l \in \Z \}.
\end{equation}

This gives a $\Z \times \Z_2$-grading on the Lie superalgebra,
i.e.
\begin{equation}\label{glgrading2} [\gl(V)(k),
\gl(V)(l)] \subset \gl(V)(k + l),
\end{equation}
and
\begin{equation}\label{glgrading3}
\gl(V)(k)= \gl(V)(k)_{\bar{0}} \oplus \gl(V)(k)_{\bar{1}},
\end{equation}
where $\gl(V)(k)_i = \gl(V)_i \cap \gl(V)(k)$.

We have
\[
X \in \gl(V)(2)
\]
by (\ref{eq:grading-gl-1}). This implies easily that
\[
\gl(V)_X = \bigoplus_{k \in \Z} \gl(V)_X(k) \quad \text{ where
}\gl(V)_X(k)= \gl(V)_X \cap \gl(V)(k).
\]

\begin{example}
Let $V$ be a superspace of $\sudim V = 3|2$. So $\gl(V)$ is
isomorphic to the Lie superalgebra $\gl(3|2)$ consisting of all
$(3+2) \times (3+2)$ square matrices. Let $X$ be a nilpotent
element in $\ev{\gl(V)}$ corresponding to the pair of partitions
$(3;2)$. Then there exist $v_1 \in \ev V$ and $u_1 \in \od V$ such
that $\{ X^2 v_1, X v_1 , v_1\}$ and $\{X u_1 , u_1 \}$ form bases
of $\ev V$ and $\od V$ respectively. Under this basis of $V$, the
element $X$ has matrix form
\begin{equation*}
\left[
\begin{array}{ccc|cc}
0 &1 & & & \\
&0 &1 & &\\
& & 0 & &\\
\hline
& & &0 &1\\
& & & &0
\end{array}
\right]
\end{equation*}
The grading $V = V(-2)\oplus V(-1) \oplus V(0) \oplus V(1) \oplus
V(2)$ is specified as follows:
\begin{align*}
{V(2)} = KX^2v_1,\quad & {V(1)}=KXu_1, \quad  {V(0)} = KX v_1,  \\
{V(-1)} =K u_1, \quad & {V(-2)} = Kv_1, \quad V(l) = 0 \text{ for
} |l| > 2.
\end{align*}

Any element $Z \in \gl(V)_X$ is of the matrix form
\[
Z=\left[
 \begin{array}{ccc|cc}
x_0 & x_2 & x_4 & y_1 & y_3 \\
0_{-2}      & x_0 & x_2 & 0_{-1}      & y_1 \\
0_{-4}      & 0_{-2}      & x_0 & 0_{-3}      & 0_{-1}      \\
\hline
0_{-1}      & z_1 & z_3 & w_0 & w_2 \\
0_{-3}      & 0_{-1}      & z_1 & 0_{-2}      & w_0
\end{array}
\right],
\]
where $x_i, y_i, z_i$ and $w_i$ are arbitrary scalars in $K$, $0_i =
0$, and the index $i$ indicates the $\Z$-gradings of the
corresponding matrix entries. Hence, $\sudim \gl(V)_X = 5 |4$.

Note that the centralizer $\gl(V)_X$ is concentrated on the
non-negative degrees.


\end{example}

\subsection{The $\Z$-gradings for $\mathfrak{osp}$}\label{sec:osp-grading}

Let $\phi$ be a nondegenerate even supersymmetric bilinear form on
$V = V_{\bar{0}} \oplus V_{\bar{1}}$, that is, $V_{\bar{0}}$ and
$V_{\bar{1}}$ are orthogonal, the restriction $\phi_{\bar{0}}$ of
$\phi$ to $V_{\bar{0}}$ is symmetric, and the restriction
$\phi_{\bar{1}}$ of $\phi$ to $V_{\bar{1}}$ is skew-symmetric. Then
the ortho-symplectic Lie superalgebra $\g =\osp (V)$ with $\g =\ev\g
+\od \g$  is defined by
\begin{equation}\label{osp1}
\g_i = \{f \in \gl(V)_i | \; \phi(f(x), y) = -(-1)^{i|x|}\phi(x,
f(y)) \; \forall \; x,y \in V\}, \quad i \in \Z_2.
\end{equation}
According to \cite{Kac}, the ortho-symplectic Lie superalgebras
$\osp (V)$ are further divided into infinite series of type $B(m,
n), C(n)$, and $D(m, n)$, depending on whether or not $\sudim V$
is $(2m + 1)| 2n, 2 | 2n,$ and $2m | 2n$, respectively.

Let $\chi \in \ev \g^*$ be a nilpotent $p$-character, which can be
regarded as $\chi \in \g^*$ by declaring that $\chi (\od \g) =0$.
Let $X \in \g_{\bar{0}}$ be the nilpotent element corresponding to
$\chi$ via the form $(\cdot,\cdot)$ on $\g$, and write it as $X =
X_0 + X_1$ where $X_i \in\gl(V_i) \cap \ev \g$. The $X_i$
determines the partition $\pi_i = (\lambda^i_1, \ldots,
\lambda^i_{r_i})$, and $\pi_i$ is the shape of the Jordan
canonical form of $X_i$.
According to \cite[Chap~1]{Jan2},
there exist bases $v_1, \ldots, v_{r_0} \in V_{\bar 0}$ and $u_1,
\ldots, u_{r_1} \in V_{\bar 1}$ satisfying properties in
\cite[1.11, Theorems~1 and 2]{Jan2} respectively.

We retain the $\Z$-grading on $V$ as defined in
Section~\ref{sec:gl-grading} in the case of $\gl(V)$ so that
(\ref{eq:grading-gl-0}) and (\ref{eq:grading-gl-1}) still hold.
Using the arguments in \cite[3.4]{Jan2}, one can show that the
grading on $V$ is compatible with the bilinear form $\phi$ in the
following sense:
\begin{equation}\label{eq:grading-osp-1}
\phi(V(k), V(l)) = 0 \text{ unless } k+l=0.
\end{equation}

We claim that
\begin{equation}\label{eq:grading-osp-2}
\g = \bigoplus_{k \in \Z} \g(k) \quad \text{where } \g(k):= \g
\cap \gl(V)(k).
\end{equation}
Indeed, let us take $Y \in \od \g$ (the argument for even elements
is similar and thus skipped) and decompose $Y = \sum_k Y_k$ with
$Y_k \in \od{(\gl(V)(k))}$. We have to show that $Y_k \in \g$,
i.e. $\phi(Y_k(v), w)= -(-1)^{|v|} \phi(v, Y_k(w))$ for $v \in
V(s)$ and $w \in V(t)$. Note that $Y_k(v) \in V(k+s)$ and $Y_k(w)
\in V(t + k)$. Now (\ref{eq:grading-osp-1}) implies
\[
\phi(Y_k(v), w) = 0 =  \phi(v, Y_k(w)) \quad \text{if }s + k + t
\neq 0.
\]
On the other hand, if $s+k+t=0$, (\ref{eq:grading-osp-1}) implies
\begin{align*}
\phi(Y_k(v),w) &= \sum_{l} \phi(Y_l(v),w) = \phi(Y(v), w) =
-(-1)^{|v|} \phi(v, Y(w))\\
&=-(-1)^{|v|} \sum_l \phi(v, Y_l(w)) = -(-1)^{|v|} \phi(v,
Y_k(w)).
\end{align*}
This proves (\ref{eq:grading-osp-2}).

By definition, $X \in \g(2)$, i.e. (\ref{grading1}) holds. This
implies, as in Section~\ref{sec:gl-grading}, that
\[
\g_X = \bigoplus_{k \in \Z} \g_X (k)
\]
where $\g_X (k) = \g_X  \cap \g(k)$. This verifies
(\ref{grading3}).

\begin{remark}
Define a homomorphism of algebraic groups $\tau = \tau_0 \times
\tau_1$ from the multiplicative group $K^{\times}$ to $GL(V)_{\bar
0} = GL(V_{\bar 0}) \times GL(V_{\bar 1})$ as follows. For each $t
\in K^\times$, let $\tau(t) = \tau_0 (t) \times \tau_1 (t)$ be the
linear map with $\tau_i(t)(v) = t^k v$ for all $v \in V(k)_i$ and
all k. Now the equality (\ref{eq:grading-gl-0}) shows the $\tau$
takes values in $SL(\ev V) \times SL(\od V)$. In the setup of
Section~\ref{sec:osp-grading}, one checks easily using
(\ref{eq:grading-osp-1}) that $\tau(K^{\times})$ is contained in
$SO(\ev V) \times Sp(\od V)$. Now the $\Z$-gradings defined in
Sections~\ref{sec:gl-grading} and \ref{sec:osp-grading} can be
described via
\begin{align*}
\gl(V)(k) &= \{ Z \in \gl(V) \mid \text{Ad}(\tau(t))(Z) = t^k
Z \text{ for all } t \in K^{\times}\}, \\
\osp (V)(k)  &= \{ Z \in \osp(V) \mid \text{Ad}(\tau(t))(Z) = t^k
Z \text{ for all } t \in K^{\times}\}.
\end{align*}
\end{remark}

\subsection{Completing the proof of Theorem~\ref{th:grading} for
$\gl$, $\mathfrak{sl}$ and $\osp$} \label{sec:gl-osp-grading-end}

Throughout this Subsection, it is assumed that $\g =\gl(V)$ or
$\osp (V)$ as in Subsection \ref{sec:gl-grading} or
\ref{sec:osp-grading}.

The invariance of the even nondegenerate bilinear form
$(\cdot,\cdot)$ under all $\tau(t)$ with $t \in K^\times$ implies
that $(\g(k), \g(l)) = 0$ if $ k+ l \neq 0$, while $(\cdot,\cdot)$
induces a perfect pairing between $\g(k)_i$ and $\g(-k)_i$ for $i
\in \Z_2$. Thus (\ref{grading2}) has been verified.

We now prove (\ref{grading4}). Let $0\neq Z \in \g_X (s)_{\bar 0}$
(the argument for $Z$ odd is similar) and write $Z(v_i)$ and
$Z(u_j)$ as in (\ref{eq:cen-gl-ev1}) and (\ref{eq:cen-gl-ev2})
respectively. We have $Z(v_i) \in V(s+1-\lambda^0_i)_{\bar 0}$ and
$Z(u_j) \in
V(s+1-\lambda^1_j)_{\bar 1}$. 
Therefore the coefficients $a_{k,l;i}$ in (\ref{eq:cen-gl-ev1})
and $b_{k,l;j}$ in (\ref{eq:cen-gl-ev2}) are $0$ unless $s + 1 -
\lambda^0_i = 2k + 1 - \lambda^0_l$ and $s + 1 - \lambda^1_j = 2k
+ 1 - \lambda^1_l$ respectively. On the other hand, we have $k
\geq \max(0, \lambda^0_l-\lambda^0_i)$ and $k \geq \max(0,
\lambda^1_l - \lambda^1_j)$ by (\ref{eq:cen-gl-ev1}) and
(\ref{eq:cen-gl-ev2}). So $a_{k,l;i} \neq 0$ (respectively,
$b_{k,l;j} \neq 0$) implies
\[
s = 2k + \lambda^0_i - \lambda^0_l \geq k \geq 0 \quad
(\text{respectively, } s = 2k + \lambda^1_j - \lambda^1_l \geq k
\geq 0).
\]
In either case, we conclude that $s \ge 0$, whence
(\ref{grading4}).

To complete the proof of Theorem~\ref{th:grading} for $\g$, it
remains to verify the identity (\ref{grading5}). Using the
argument of \cite[Lemma~5.7]{Jan2}, one can show that
\begin{equation}\label{eq:grading-general-1}
[\g(k-2)_i, X]=\g(k)_i \Longleftrightarrow \g_X  \cap \g(-k) =0
\end{equation}
for each $k \in \Z$ and $i \in \Z_2$. Now by (\ref{grading4}) and
(\ref{eq:grading-general-1}), the map $\ad X: \g(k)_i \ra
\g(k+2)_i$ is surjective with kernel the subspace $\g(k)_i \cap
\g_{X,i} $ for $k \geq 0$ and $i \in \Z_2$. The highest degree
component in the grading is contained in $\g_X$; counting the
dimension of $\g_X$ backwards through the grading, we prove
\eqref{grading5}.

This completes the proof of Theorem~\ref{th:grading} for $\g
=\gl(V)$ or $\osp (V)$.

\begin{remark}
For $\g =\mathfrak{sl}(m|n)$ with $p \nmid m-n$ and $p>2$, the
arguments in Subsections \ref{sec:gl-grading},
\ref{sec:osp-grading} and the above carry over readily, and so
Theorem~\ref{th:grading} holds for $\mathfrak{sl}(m|n)$ as well.
\end{remark}

\subsection{The $\Z$-gradings for exceptional Lie superalgebras}

Now suppose that $\g$ is one of excetional Lie superalgebras of
type $D(2, 1; \alpha), F(4),$ or $G(3)$ as in the Table of
Section~\ref{sec:bcLsa-p}, which admit a nondegenerate even
supersymmetric bilinear form and whose even subalgebras are Lie
algebras of reductive algebraic groups. In each case, the lower
bound for the characteristic of $K$ is indeed $\geq 3(h-1)$, where
$h$ is the maximum of the Coxeter numbers for the irreducible
summands of the even subalgebra $\ev \g$. As before let $X\in \g$
denote the nilpotent element corresponding to a nilpotent
$p$-character $\chi \in \ev \g^*$. With this restriction on $p$,
by Carter \cite{Ca} (also cf. \cite[Sect.~5.5]{Jan2}), there
exists an $\mathfrak{sl}(2)$-triple $\{X, H, Y\}$ so that $KH$
lifts to a one-dimensional torus $ \gamma : K^{\times} \ra \ev G $
in the simply connected algebraic group associated to $\ev\g$
which induces a $\Z$-grading
\begin{equation*}\label{FGD1}
\g = \oplus_{k\in \Z} \g(k), \quad \text{where }  \g(k)= \{z \in
\g \mid \text{Ad}(\gamma(t))(z) = t^k z \text{ for all } t \in
K^{\times} \}.
\end{equation*}
Moreover, $\ev\g$ is semisimple under the adjoint
$\mathfrak{sl}(2)$-action. By the explicit information on $\od\g$
given in Subsection~\ref{sec:bcLsa-p}, we further observe that
$\od\g$ is also semisimple under the adjoint
$\mathfrak{sl}(2)$-action. Now it follows that this grading
satisfies all the properties (\ref{grading1})--(\ref{grading5}) in
Theorem~\ref{th:grading} just as for semisimple Lie algebras in
characteristic zero (compare with \cite{Jan2}).

\section{Proof of Super KW Conjecture with nilpotent $p$-characters}
\label{sec:proof}

\subsection{The subalgebra $\mathfrak m$}  \label{sec:subalgm}

Let $\g$ be one of the basic classical Lie superalgebras in
Sect.~\ref{sec:bcLsa-p} with a non-degenerate supersymmetric even
bilinear form denoted by $(\cdot, \cdot)$. Furthermore, we have a
$\Z$-grading $\g = \oplus_{k \in \Z}\g(k)$ with favorable
properties as in Theorem~\ref{th:grading}. It follows that for
each $k$ there is a non-degenerate pairing between $\g(k)$ and
$\g(-k)$ by (\ref{grading2}) and that $\sudim \g(k) = \sudim
\g(-k)$. On $\g(-1)_{\bar 0}$ (resp. $\g(-1)_{\bar 1}$) there is a
symplectic (resp. symmetric) bilinear form $\langle \cdot,
\cdot\rangle$ given by
\begin{equation}\label{ssform}
\langle x, y \rangle := (X, [x, y]) = \chi([x, y]).
\end{equation}
This form is even non-degenerate and skew-supersymmetric. Indeed,
take a nonzero $x \in \g(-1)_i$ for $i \in \Z_2$. Then it follows
by (\ref{grading4}) that $0 \neq [X, x] \in \g(1)_i$. By the
non-degeneracy of the pairing between $\g(1)_i$ and $\g(-1)_i$,
there exists $y \in \g(-1)_i$ with $0 \neq ([X,x], y) = (X, [x,
y]) = \langle x, y \rangle$.

Note that $\dim \g(-1)_{\bar 0}$ is even. Take $\g(-1)'_{\bar 0}
\subset \g(-1)_{\bar 0}$ to be a maximal isotropic subspace with
respect to $\langle\cdot, \cdot\rangle$. It satisfies $\dim
\g(-1)'_{\bar 0} = \dim \g(-1)_{\bar 0}/2$.

Denote $r = \dim \g(-1)_{\bar 1}$. There is a basis $v_1,\ldots,
v_r$ of $\g(-1)_{\bar 1}$ under which the symmetric form
$\langle\cdot, \cdot\rangle$ has matrix form
\begin{equation*}
\left[
\begin{array}{ccc}
 & & 1\\
 & \Ddots &\\
 1 & &
\end{array}
\right].
\end{equation*}
If $r$ is even, take $\g(-1)'_{\bar 1} \subset \g(-1)_{\bar 1}$ to
be the subspace spanned by $v_1, \ldots, v_{\frac{r}{2}}$. If $r$
is odd, take $\g(-1)'_{\bar 1} \subset \g(-1)_{\bar 1}$ to be the
subspace spanned by $v_1, \ldots, v_{\frac{r-1}{2}}$. Set
$\g(-1)'=\g(-1)'_{\bar 0} \oplus \g(-1)'_{\bar 1}$ and introduce
the subalgebras
\begin{align*}
\m &= \bigoplus_{k \geq 2} \g(-k) \oplus \g(-1)',\\
\m' &=\begin{cases} \m \oplus Kv_{\frac{r+1}{2}}, & \text{for $r$
odd}\\
\m, & \text{for $r$ even}.
\end{cases}
\end{align*}

The subalgebra $\m$ is $p$-nilpotent, and the linear function
$\chi$ vanishes on the $p$-closure of $[\m, \m]$. It follows by
Proposition~\ref{!simple} that $U_{\chi}(\m)$ has the trivial
module as its only irreducible module and $U_{\chi}(\m)/N_{\m} =
K$ where $N_{\m}$ is the Jacobson radical of $U_{\chi}(\m)$.

Assume that $r$ is odd. The induced $U_{\chi}(\m')$-module $V =
U_{\chi}(\m') \otimes_{U_{\chi}(\m)} K$ is two-dimensional,
irreducible, and admits an odd automorphism of order $2$ induced
from  $v_{\frac{r+1}{2}}$. By Frobenius reciprocity, it is the
only irreducible $U_{\chi}(\m')$-module. Denote by
$\mathfrak{q}_d(K)$ the simple superalgebra of type $Q$ consisting
of all $2d \times 2d$ matrices of the form
\begin{equation*}
\left[
\begin{array}{rr}
A & B\\
-B & A
\end{array}
\right]
\end{equation*}
with $A$ and $B$ arbitrary $d \times d$ matrices. We summarize the
above as follows.

\begin{proposition} \label{prop:q1}
Assume that $\dim \g(-1)_{\bar 1}$ is odd. Then $U_{\chi}(\m')$
has a unique simple module; it is isomorphic to $V$, which is
$2$-dimensional and of type $Q$. Moreover,  $U_{\chi}(\m')/
N_{\m'}$ is isomorphic to the simple superalgebra
$\mathfrak{q}_1(K)$.
\end{proposition}

We have the following commutative diagram:
\[
\begin{array}{ccc}
\rea{\chi}{\m} & \subset & \rea{\chi}{\m'} \\
\cup &  & \cup\\
N_{\m} &= & N_{\m'}
\end{array}
\]

\subsection{Freeness over $U_{\chi}(\m)$}

For a $U_{\chi}(\g)$-module $M$ put
\begin{equation}\label{eq:ann1}
M^{\m} = \{ v \in M | N_{\m} \cdot v = 0 \} =\{ v \in M | N_{\m'}
\cdot v = 0 \} .
\end{equation}

\begin{proposition}\label{prop:freeness}
Let $\g$ be one of the basic classical Lie superalgebras. Then
every $U_{\chi}(\g)$-module is $U_{\chi}(\m)$-free.
\end{proposition}
\begin{proof}
Since our proof is a straightforward superalgebra generalization
of the proof of \cite[Theorem~1.3]{Skr}, we only formulate the
main steps below and refer to {\em loc. cit.} for details.

Recall the $\Z$-grading on $\g=\oplus_{k\in \Z}\g(k)$ from
Sect.~\ref{sec:grading}. Define a decreasing filtration $\{ \g^k
\}$ of $\g$ by setting $\g^k = \sum_{l \geq k} \g(-l)$ with
induced $\Z_2$-grading $\g^k = \g^k_{\bar 0} \oplus \g^k_{\bar 1}$
for $k \in \Z$. It clearly satisfies the conditions in
\cite[(b1)--(b6), pp.568]{Skr}.
%
%
%
%


First of all, we have $[\g^2, N_{\m}] \subset N_{\m}$, and so
$M^{\m}$ is stable under $\g^2$ for every $U_{\chi}(\g)$-module
$M$.

Let $y_{0,1}, \ldots, y_{0,r_0}$ and $y_{1,1}, \ldots, y_{1,r_1}$
be the respective bases for $\m_{\bar 0}$ and $\m_{\bar 1}$
compatible with the $\Z$-grading on $\g$. Define $d_{0,s}
> 0$ (resp. $d_{1,s} > 0$) by the condition $y_{0,s} \in
\g^{d_{0,s}}_{\bar 0} \setminus \g^{d_{0,s} + 1}_{\bar 0}$ (resp.
$y_{1,s} \in \g^{d_{1,s}}_{\bar 1} \setminus \g^{d_{1,s} +
1}_{\bar 1}$).
%
Let $\Xi = \{ \mathbf{a}= (a_1, \ldots, a_{r_0}; \epsilon_1,
\ldots, \epsilon_{r_1}) \vert \; 0 \leq a_s < p, \, \epsilon_t \in
\{0, 1\} \}$. For $\mathbf{a} \in \Xi$ put
\begin{gather}
|\mathbf{a}| = \sum a_s + \sum \epsilon _t ,  \notag \\
\text{wt}\, \mathbf{a} = \sum d_{0,s} a_s + \sum d_{1,t} \epsilon_t , \notag \\
y_{\mathbf{a}} = (y_{0,1} - \eta(y_{0,1}))^{a_1}\cdots (y_{0,r_0}
- \eta(y_{0,r_0}))^{a_{r_0}} y_{1,1}^{\epsilon _1} \cdots
y_{1,r_1}^{\epsilon _{r_1}} \in U_{\chi}(\m) , \notag
\end{gather}
where $\eta \in \ev \m^*$ is the linear function defining the
one-dimensional representation on $\m$ with $p$-character
$\chi|_{\m}$.

For each pair $i, j$ denote $\Lambda(i, j) = \{\mathbf{b} \in \Xi
\mid \text{wt} \,\mathbf{b} = i \text{ and } |\mathbf{b}| = j \}$,
and consider any linear ordering $\preceq$ on $\Xi$ subject to the
following condition: $\mathbf{a} \prec \mathbf{b}$ whenever either
$\text{wt}\, \mathbf{a} < \text{wt}\, \mathbf{b}$ or $\text{wt}\,
\mathbf{a} = \text{wt}\, \mathbf{b}$, $|\mathbf{a}| >
|\mathbf{b}|$. Let $M$ be a $\rea{\chi}{\g}$-module. Following
\cite[proof of Theorem~1.3]{Skr}, for each $v \in M^{\m}$ and
$\mathbf{b} \in \Lambda(i, j)$, there exists $v_{\mathbf{b}} \in
M$ such that
\begin{equation}\label{eq:filtration5}
y_{\mathbf{a}} v_{\mathbf{a}} = v,  \text{ and } y_{\mathbf{a}}
v_{\mathbf{b}} = 0 \text{ when } \mathbf{a} \succ \mathbf{b}.
\end{equation}

Let $\m$ operate in $E = \Hom_K (U_{\chi}(\m), M^{\m})$ by
$(yf)(u) = (-1)^{|y|(|f| + |u|)}f(uy)$ for homogeneous $y \in \m$
, $f \in E$, and $u \in U_{\chi}(\m)$. Take any even linear map
$\pi : M \rightarrow M^{\m}$ such that $\pi |_{M^{\m}} =
\text{id}$ and define $\varphi : M \rightarrow E$ setting
$\varphi(w)(u) = (-1)^{|u||w|}\pi(uw)$ for $w \in M$ and $u \in
U_{\chi}(\m)$. Clearly $\varphi$ is a homomorphism of
$U_{\chi}(\m)$-modules. By (\ref{eq:filtration5})
$\varphi(v_{\mathbf{a}})(y_{\mathbf{a}}) = v$ and
$\varphi(v_{\mathbf{b}})(y_{\mathbf{a}}) = 0$ when $\mathbf{a}
\succ \mathbf{b}$ for any $v \in M^{\m}$. Given $\psi \in E$, it
can be shown by downward induction on $\mathbf{b}$ that there
exists $w \in M$ such that $\varphi(w)(y_{\mathbf{a}}) =
\psi(y_{\mathbf{a}})$ for all $\mathbf{a} \succeq \mathbf{b}$.
Since the elements $y_{\mathbf{a}}$ form a basis for
$U_{\chi}(\m)$, the map $\varphi$ is surjective. Furthermore,
$\text{ker}(\varphi)$ is an $\m$-submodule of $M$ which has zero
intersection with $M^{\m}$ because $\varphi(v)(1) = v$ for every
$v \in M^{\m}$. Then $\text{ker} \varphi = 0$. Thus $\varphi$ is
an isomorphism of $U_{\chi}(\m)$-modules. Since $U_{\chi}(\m)$ is
a Frobenius algebra by Proposition~\ref{prop:Fro}, the
$U_{\chi}(\m)$-module $E \cong U_{\chi}(\m)^* \otimes M^{\m}$ is
free. Hence $M$ is $U_{\chi}(\m)$-free.
\end{proof}

\subsection{The Super KW Property with nilpotent $p$-characters}

\begin{theorem}\label{th:KW-N}
Let $\g$ be one of the basic classcial Lie superalgebras as in
Section~\ref{sec:bcLsa-p}, and let $\chi \in \ev \g^*$ be
nilpotent. Let $d_i = \dim \g_i - \dim \g_{\chi,i},\; i \in \Z_2$.
Then the dimension of every $U_{\chi}(\g)$-module $M$ is divisible
by $p^{\frac{d_0}{2}} 2^{\lfloor \frac{d_1}{2}\rfloor}$.
\end{theorem}
Note that $p^{\frac{d_0}{2}} 2^{\lfloor \frac{d_1}{2}\rfloor}
=\dim \rea{\chi}{\m'}$ ($\neq \dim \rea{\chi}{\m}$ in general).

\begin{proof}
By (\ref{grading5}) and since $\sudim \g(k) = \sudim \g(-k)$, we
have
\begin{equation}\label{m-dim}
\sudim \g - \sudim \g_{\chi} = \sum_{k \geq 2} 2 \sudim \g(-k) +
\sudim \g(-1).
\end{equation}
In particular, $\dim \g(-1)_{\bar 1}$ and $d_1$ have the same
parity. It follows now from the definition of $\m$ that either (1)
$\frac{d_0}{2} | \frac{d_1}{2} = \sudim \m$ when $\dim
\g(-1)_{\bar 1}$ and $d_1$ are even, or (2) $\frac{d_0}{2} |
\frac{d_1 - 1}{2} = \sudim \m$ when $\dim \g(-1)_{\bar 1}$ and
$d_1$ are odd.

In case (1), the theorem follows immediately from
Proposition~\ref{prop:freeness}.

In case (2), for each $U_{\chi}(\g)$-module $M$, $M^{\m}$ is a
module over the superalgebra $U_{\chi}(\m')/ N_{\m} \cong
\mathfrak{q}_1(K)$, by Proposition~\ref{prop:q1}. Since the
(unique) simple module of $\mathfrak{q}_1(K)$ is two-dimensional,
$M^{\m}$ has dimension divisible by $2$. Now the isomorphism $M
\cong U_{\chi}(\m)^* \otimes M^{\m}$ (cf.
Proposition~\ref{prop:freeness} and its proof) implies the desired
divisibility.
\end{proof}

Theorem~\ref{th:KW-N} can be somewhat strengthened in the
following form.
\begin{theorem} \label{th:finW}
Set $\delta = \dim \rea{\chi}{\m}$ and denote by
$\mathcal{Q}_{\m}$ the induced $\rea{\chi}{\g}$-module
$\rea{\chi}{\g}\otimes_{\rea{\chi}{\m}}K_{\chi}$. Then
$\mathcal{Q}_{\m}$ is a projective $\rea{\chi}{\g}$-module and
\[
\rea{\chi}{\g} \cong M_{\delta} (W_\chi(\g)^{op}), \quad
\text{where }
W_\chi(\g)=\text{End}_{\rea{\chi}{\g}}(\mathcal{Q}_{\m}).
\]
\end{theorem}

\begin{proof}
Let $V_1, \ldots, V_s$ (resp. $W_1, \ldots, W_t$) be all
inequivalent simple $\rea{\chi}{\g}$-modules of type $M$ (resp. of
type $Q$). Let $P_i$ (resp. $Q_j$) denote the projective cover of
$V_i$ (resp. $W_j$). By Proposition~\ref{prop:freeness}, $V_i$ and
$W_j$ are free over $\rea{\chi}{\m}$. It follows by Frobenius
reciprocity that
\[
\dim \Hom_{\g}(\mathcal{Q}_{\m}, V_i)= \dim \Hom_{\m}(K_{\chi},
V_i) = :a_i.
\]
By Frobenius reciprocity and Proposition~\ref{!simple},
\[
\dim \Hom_{\g}(\mathcal{Q}_{\m}, W_j)= \dim \Hom_{\m}(K_{\chi},
W_j) = \dim \Hom_{\m}(\rea{\chi}{\m}, W_j)
\]
which has to be an even number, say $2b_j$, since as a type $Q$
module $W_j$ admits an odd involution commuting with $\m$. It
follows that the ranks of the free $\rea{\chi}{\m}$-modules $V_i$
and $W_j$ are $a_i$ and $2b_j$ respectively. Put
\[
P = \bigoplus_{i=1}^{s}  P_i^{a_i} \bigoplus \bigoplus_{j=1}^t
Q_j^{b_j}.
\]
Then $P$ is projective and has the same head as
$\mathcal{Q}_{\m}$. So there is a surjective homomorphism $\psi: P
\ra \mathcal{Q}_{\m}$.

Since $\dim V_i = \delta a_i$ and $\dim W_j = 2\delta b_j$, by
Wedderburn theorem for superalgebras (cf. Kleshchev
\cite[Theorem~12.2.9]{Kle}) the left regular
$\rea{\chi}{\g}$-module is isomorphic to $P^{\delta}$. The
equality of dimensions
\[
\dim P = \dim \rea{\chi}{\g}/{\delta} = \dim \mathcal{Q}_{\m}
\]
implies that $\psi$ is an isomorphism. Finally,
\begin{eqnarray*}
\rea{\chi}{\g} &\cong&
\text{End}_{\rea{\chi}{\g}}(\rea{\chi}{\g})^{op} \cong
\text{End}_{\rea{\chi}{\g}}(P^{\delta})^{op}
\\
&\cong& (M_{\delta} (\text{End}_{\rea{\chi}{\g}}(P)))^{op} \cong
M_{\delta} (W_\chi(\g)^{op}).
\end{eqnarray*}
This completes the proof of the theorem.
\end{proof}

\begin{remark}
In the case $\g$ is a Lie algebra (i.e. $\ev \g =0$), we recover a
theorem of Premet~\cite[Theorem~2.3 (i),(ii)]{Pr2} with somewhat
modified arguments which avoid the use of support variety
machinery.

The algebra $W_\chi(\g)$ has a counterpart over $\C$, which is
usually referred to as a finite $W$-superalgebra in the math
physics literature.
\end{remark}
\begin{remark} \label{rem:semisimple}
Theorems~\ref{th:KW-N} and \ref{th:finW} remain valid when $\g$ is
a direct sum of basic classical Lie superalgebras.
\end{remark}

\section{A reduction from general to nilpotent $p$-characters}\label{sec:gen. char.}

In this section, we will establish a Morita equivalence which
reduces the case of a general $p$-character to a nilpotent one,
completing the proof of the Super KW Conjecture for $\g$.

\subsection{An equivalence of categories} \label{sec:gen. char. ass.}

Let $\g$ be a basic classical Lie superalgebra as in
\ref{sec:bcLsa-p}.

Let $\chi = \chi_s + \chi_n$ be the Jordan decomposition of $\chi
\in \ev \g^*$ (we regard $\chi \in \g^*$ by letting $\chi(\od \g)
=0$). Under the isomorphism $\ev\g^* \cong \ev\g$ induced by the
nondegenerate bilinear form $(\cdot,\cdot)$ on $\ev\g$, this can
be identified with the usual Jordan decomposition $s+n$ on
$\ev\g$. Take a Cartan subalgebra $\h$ of $\g$ which contains $s$,
and recall that $\g$ admits a root space decomposition (cf.
\cite{Kac})
$$\g = \h \bigoplus \bigoplus_{\alpha \in \Phi} \g_\alpha.$$
Then it follows that $\g_s =: \mathfrak{l} =\ev {\mathfrak l}
\oplus \od {\mathfrak l}$ also has a root space decomposition
$$\mathfrak l =\h \bigoplus \bigoplus_{\alpha \in \Phi(\mathfrak{l})} \g_\alpha,$$
where $\Phi(\mathfrak{l}) =\{\alpha\mid \alpha(s) =0\}$.


It is a well-known super phenomenon that all the systems of simple
roots (resp. the systems of positive roots, or resp. Borel
subalgebras) of $\g$ are not equivalent under the Weyl group $W$.
An explicit list of non-$W$-equivalent systems of positive roots
can be found in Kac \cite[pp.51--53]{Kac}.

The following proposition is proved by case-by-case calculations,
which is completely elementary yet tedious and thus omitted. From
the detailed calculation we find that the following system of
simple roots for $F(4)$ up to $W$-equivalence
$$\{-\delta, \hf (\delta -\varepsilon_1 -\varepsilon_2 +\varepsilon_3),
\varepsilon_1 -\varepsilon_2, \varepsilon_2 -\varepsilon_3\}$$ is
missing from Kac's list. (We learned that this was also noticed by
Serganova.)

\begin{proposition}
Let  $\mathfrak{l} = \g_{s}$ with $s$ in a Cartan subalgebra $\h$
of $\g$. There exists a system $\Pi$ of simple roots of $\g$ such
that $\Pi \cap \Phi(\mathfrak{l})$ is a system of simple roots for
$\Phi(\mathfrak{l})$. In particular $\mathfrak l$ is always a
direct sum of basic classical Lie superalgebras.

Let $\mathfrak{b} =\h \oplus \mathfrak n$ be the Borel subalgebra
associated to $\Pi$. Then we can define a parabolic subalgebra $\p
= \mathfrak{l} + \mathfrak{b} = \mathfrak{l} \oplus \mathfrak{u}$,
where $\mathfrak{u}$ denotes the nilradical of $\p$.
\end{proposition}

Note that $\chi(\mathfrak u) =0$ since $s+n \in \g_\chi \subset
\g_{\chi_s} =\mathfrak l$, and $\chi(\mathfrak u) =(s+n, \mathfrak
u) \subset (\mathfrak l, \mathfrak u) =0$. Also note that
$\chi_s|_{\mathfrak l} =0$ and hence $\chi|_{\mathfrak l}
=\chi_n|_{\mathfrak l}$ is nilpotent. Thus any
$\rea{\chi}{\mathfrak l}$-module can be regarded as a
$\rea{\chi}{\p}$-module with a trivial action of $\mathfrak u$.
Here and below, by abuse of notation, we will use the same letter
$\chi$ for its restrictions on $\p$, $\mathfrak l$, or $\mathfrak
u$.

Given an associative $k$-superalgebra $A$, $A$-$\mathfrak{mod}$
denotes the category of finite-dimensional $A$-(super)modules with
{\em even} morphisms, which turns out to be an abelian category.
One can easily switch back and forth by a parity functor between
this category and the category of finite-dimensional
$A$-(super)modules with {\em arbitrary} morphisms (cf. e.g.
\cite{BKN}).

\begin{theorem}\label{th:morita}
Let $\g$ be a basic classical Lie superalgebra as in Section
\ref{sec:bcLsa-p}. Let $\chi =\chi_s +\chi_n\in \ev\g^*$ be a Jordan
decomposition, $\mathfrak{l} =\g_{\chi_s}$, and $\p =\mathfrak{l}
\oplus \mathfrak u$ as above. Then $-^{\mathfrak u}:
\rea{\chi}{\g}$-$\mathfrak{mod} \rightarrow \rea{\chi}{\mathfrak
l}$-$\mathfrak{mod}$ is an equivalence of categories, and its
inverse is given by $\rea{\chi}{\g} \otimes_{\rea{\chi}{\p}} -:
\rea{\chi}{\mathfrak l}$-$\mathfrak{mod} \rightarrow
\rea{\chi}{\mathfrak g}$-$\mathfrak{mod}$. Moreover,
$\rea{\chi}{\g}$ and $\rea{\chi}{\mathfrak l}$ are Morita
equivalent.
\end{theorem}
The above theorem is a super analogue of a theorem of
Friedlander-Parshall \cite[Theorem~3.2]{FP} which was in turn
built on the earlier work of Kac-Weisfeiler \cite{KW}. By the same
argument as in \cite[pp.~1068]{FP}, we reduce the proof of
Theorem~\ref{th:morita} to the following theorem, which is a super
analogue of the main theorem of \cite{KW} and
\cite[Theorem~8.5]{FP}.

\begin{theorem}\label{th:ME-simple}
Retain the above notation. Then for any irreducible
$\rea{\chi}{\g}$-module $M$, $M^{\mathfrak{u}}$ is an irreducible
$\rea{\chi}{\p}$-module and the natural map
\[
\rea{\chi}{\g} \otimes_{\rea{\chi}{\p}} M^{\mathfrak{u}}
\rightarrow M
\]
is an isomorphism. Also, $M$ is a projective
$\rea{\chi}{\mathfrak{u}}$-module.
\end{theorem}

\subsection{Proof of Theorem~\ref{th:ME-simple}}

The proof of Theorem~\ref{th:ME-simple} follows the same strategy
as in the Lie algebra case given in \cite[Sect.~8]{FP}, with a few
modifications. In the following, we only formulate and establish
the parts which differ more substantially, while omitting the
parts which are completely analogous to the Lie algebra case and
referring to {\em loc. cit.} for details.

The proof in \cite[Sect.~8]{FP} is based on four lemmas (Lemmas
8.1, 8.2, 8.3, and 8.4 therein), and it can be literally copied to
the super setup once the super analgoues of the four lemmas are
formulated and established. The super analogues of Lemmas 8.1 and
8.3 are obtained in a straightforward manner and thus skipped. On
the other hand, the super analogues of the remaining lemmas need
some extra care.

The complication in the following super analogue of
\cite[Lemma~8.2]{FP} arises from the fact that there are three
types of roots in $\g$:
\begin{enumerate}
\item[(i)] $\Phi \cap \mathbb Q \delta =\{\pm \delta\}$ with
$\delta$ even;

\item[(ii)] $\Phi \cap \mathbb Q \delta =\{\pm \delta \}$ with
$\delta$ odd;

\item[(iii)] $\Phi \cap \mathbb Q \delta =\{\pm \delta, \pm
2\delta \}$ with $\delta$ odd and $2\delta$ even.
\end{enumerate}
They correspond to three rank one Lie superalgebras
$\mathfrak{sl}(2)$, $\mathfrak{sl}(1|1)$ and  $\osp(1|2)$
respectively. For latter purpose, for such a $\delta$, we shall
denote
\begin{eqnarray*}
\delta^* = \left\{
\begin{array}{ll}
 \delta, & \text{ in case (i) and (ii)} \\
 \{ \delta, 2\delta\}, & \text{ in case (iii)}.
\end{array}
\right.
\end{eqnarray*}

\begin{lemma}\label{lem:ME2}
Let $\p =\mathfrak{l} \oplus \mathfrak{u}$ be a parabolic
subalgebra of a basic classical Lie superalgebra $\g$ containing a
Borel subalgebra $\mathfrak{b}$. Assume that the commutator
subalgebra $\mathfrak{l}' = [\mathfrak{l} ,\mathfrak{l}]$ of
$\mathfrak{l}$ is isomorphic to one of the following:
\begin{itemize}
\item[(i)] $\mathfrak{sl}(2)$ with standard basis $\{ e, f, h\}$
such that $\chi(e)=0 =\chi(f)$ and $\chi(h) \neq 0$;

\item[(ii)] $\mathfrak{sl}(1|1)$ with basis $\{X=
\begin{pmatrix} 0 & 1 \\0 &0 \end{pmatrix}, Y =
\begin{pmatrix} 0 & 0 \\1 &0 \end{pmatrix}, h = \begin{pmatrix} 1 & 0 \\0 &1
\end{pmatrix}\}$ such that $\chi(h) \neq 0$;

\item[(iii)] $\osp(1 | 2)$ with standard basis $\{ e, f, h ; E,
F\}$ (see Section~\ref{sec:osp12}) such that $\chi(e)=0 =\chi(f)$
and $\chi(h) \neq 0$.
\end{itemize}
Let $V$ be a finite dimensional $\rea{\chi}{\mathfrak{b}}$-module
upon which $e$ (resp. $E$, $X$) acts trivially. Denote by $I =
\rea{\chi}{\p} \otimes_{\rea{\chi}{\mathfrak{b}}}-:
\rea{\chi}{\mathfrak{b}}$-$\mathfrak{mod} \rightarrow
\rea{\chi}{\p}$-$\mathfrak{mod}$ the induction functor. Then
$I(V)^e = V$ (resp. $I(V)^E = V$, $I(V)^X = V$).
\end{lemma}

\begin{proof}
The key here as in the proof of \cite[Lemma~8.2]{FP} is to show
that $\rea{\chi}{\mathfrak{l}'}$ is a semisimple superalgebra.
This is well known for the $\mathfrak{sl}(2)$ case, and we will
prove it for the $\osp(1|2)$ case in Sect.~\ref{sec:osp12} (see
Proposition~\ref{prop:osp12semisimple} for a more precise
statement).

We now consider the case when $\mathfrak{l}'$ is isomorphic to
$\mathfrak{sl}(1|1)$ with $\chi(h) \neq 0$. For $\lambda$
satisfying $\lambda^p - \lambda - \chi(h)^p = 0$, the baby Verma
module
\[
Z_{\chi}(\lambda) = \rea{\chi}{\mathfrak{sl}(1|1)}
\otimes_{\rea{\chi}{KX+Kh}}K_{\lambda}
\]
is two-dimensional with a basis $\{ v_{\lambda}=1\otimes 1_\la,
Yv_{\lambda}\}$, where $K_{\lambda} =K 1_\la$ denotes the
1-dimensional representation of $KX+Kh$ with $X. 1_\la =0, h.
1_\la =\la 1_\la$. The action of $\mathfrak{sl}(1|1)$ is given by
$$
 Xv_{\lambda} = 0,\; X Y v_{\lambda} = \lambda v_{\lambda},\; h
v_{\lambda} = \lambda v_{\lambda},\; h Y v_{\lambda} = \lambda Y
v_{\lambda},\; Y Y v_{\lambda} =0.
$$
Observe that each $Z_{\chi}(\lambda)$ is irreducible of type $M$,
and $Z_{\chi}(\lambda)$ are pairwise non-isomorphic (there are $p$
of them in total). Since $\dim \rea{\chi}{\mathfrak{sl}(1|1)} =
4p$, this forces each $Z_{\chi}(\lambda)$ to be projective and
$\rea{\chi}{\mathfrak{sl}(1|1)}$ to be semisimple.

Thus each $U_\chi(\mathfrak l')$-module is projective and hence
projective as a module over its subalgebra generated by $X$ (resp.
by $E$ in the $\osp(1|2)$ case). Hence $I(V)^X$ (resp. $I(V)^E$)
has dimension equal to $\dim I(V)/2$ (resp. $\dim I(V)/2p$), which
coincides with $\dim V$.
\end{proof}

We now go back to the notation for $\chi, \mathfrak{b}$ and
$\mathfrak{p}$ as in \ref{sec:gen. char. ass.}. Let
$\mathfrak{p}^- = \mathfrak{l} \oplus \mathfrak{u}^-$ be the
parabolic subalgebra opposite to $\mathfrak{p}$. Denote
\begin{gather*}
\Phi^+_s: =\{\text{roots whose root vectors lie in } \mathfrak{n}
\cap
\mathfrak{l}\}, \quad \Phi^-_s = - \Phi^+_s;\\
\Phi_u: =\{\text{roots whose root vectors lie in } \mathfrak{u}\},
\quad  \Phi^-_u = -\Phi_u.
\end{gather*}

The following is a super analogue of \cite[Lemma~8.4]{FP} with a
different proof. The complication in the superalgebra setup arises
from the fact that there are odd roots and the longest element in
the Weyl group $W$ does not send a system of positive roots to its
opposite. We shall need the notion of odd reflections (which has
been used by Serganova and others in various situations; cf.
\cite{Ser, SW} for references).
\begin{lemma}\label{lem:ME4}
We can enumerate $\Phi_{\mathfrak{u}} =\{\delta_1^*, \ldots,
\delta_t^*\}$ as a sequence of singletons or pairs of roots so
that for each $i$,
\[
\Phi^+_i := \Phi^+_s \cup \{-\delta_1^*, \ldots, -\delta_{i-1}^*,
\delta_i^*, \ldots, \delta_t^*\}
\]
is a system of positive roots for $\Phi$ in which $\delta_i$ is a
simple root. Moreover, for each $i$, $ \Psi_i :=
\{-\delta_1^*,\ldots, -\delta_i^*\}$ is a closed subsystem of
$\Phi$ normalized by $\Phi^+_s$.
\end{lemma}
\begin{proof}
For a root $\delta$ in a set of simple roots $\tilde \Pi$
associated to the system of positive roots $\tilde{\Phi}^{+}$, let
$r_{\delta} : \Phi \rightarrow \Phi$ be the (even or odd)
reflection associated to $\alpha$. (see \cite{Ser} for the basic
properties of odd reflections). If $2\delta$ is a root, $r_\delta$
is by definition the even reflection $r_{2\delta}$. It is known
that $r_\delta \tilde{\Phi}^{+}$ is a system of positive roots,
$-\delta^* \in r_\delta \tilde{\Phi}^{+}$, and $r_\delta
\tilde{\Phi}^{+} \cap \tilde{\Phi}^{+} = \tilde{\Phi}^{+}
\backslash \delta^*$.

Denote by $\Pi_0$ the set of simple roots associated to the system
of positive roots $\Phi_0^+ =\Phi_s^- \cup \Phi_u$. Pick $\delta_1
\in \Pi_0\cap \Phi_u$. We proceed inductively. Assume that we have
defined $\Pi_0, \ldots, \Pi_{i-1}$ and have chosen $\delta_1 \in
\Pi_0\cap \Phi_u,\ldots, \delta_i \in \Pi_{i-1}\cap \Phi_u$. Put
$\Pi_i = r_{\delta_i}(\Pi_{i-1})$ and define $\Phi^+_i$ to be the
positive system determined by $\Pi_i$. Then we have
\[
\Phi_{i-1}^+ \setminus \{\delta_i^*\} = \Phi^+_i \setminus
\{-\delta_i^*\}.
\]
It follows that
\[
\Phi_{i-1}^+ \cap \Phi^+_0 =  (\Phi_i^+ \cap \Phi^+_0) \bigsqcup
\{\delta_i^*\},
\]
and thus $\vert \Phi_i^+ \cap \Phi^+_0 \vert$ is $1$ or $2$ less
than $\vert \Phi_{i-1}^+ \cap \Phi^+_0 \vert$. Repeating this
process, we obtain $\Phi^+_i$ and $\Psi_i :=\Phi^+_{i+1} \cap
\Phi^-_u$ ($i=1,\ldots, t$), until $\Phi^+_t \cap \Phi^+_0
=\emptyset$.

It follows from $\Psi_i =\Phi_{i+1}^+ \cap \Phi_u^-$ that $\Psi$
is a closed subsystem of $\Phi$. Given $\alpha \in \Phi^+_s$ so
that $\alpha -\delta_j$ (or $\alpha -2\delta_j$) is a root, then
$\alpha -\delta_j$ (or $\alpha -2\delta_j$) lies in
$\Phi_{i+1}^+$; moreover it lies in $\Psi_i =\Phi_{i+1}^+ \cap
\Phi_u^-$ since $\Phi_u^-$ is normalized by $\Phi_s^+$.
\end{proof}

\subsection{Proof of the Super KW Property for $\g$}
Now we are in a position to prove the Super KW Property with
arbitrary $p$-characters.

\begin{theorem}[Super Kac-Weisfeiler Property] \label{th:KW}
Let $\g$ be a basic classical Lie superalgebra as in
Section~\ref{sec:bcLsa-p}, and let $\chi \in \ev \g^*$. Let $d_i =
\dim \g_i - \dim \g_{\chi, i}$, $i \in \Z_2$. Then the dimension
of every $\rea{\chi}{\g}$-module $M$ is divisible by
$p^{\frac{d_0}{2}}2^{\lfloor \frac{d_1}{2} \rfloor}$.
\end{theorem}

\begin{proof}
Observe that
\begin{align*}
\sudim \g - \sudim \g_{\chi} &= \sudim \g - \sudim
\mathfrak{l}_{\chi_n} \\
& = 2 \sudim \mathfrak{u^-} + (\sudim \mathfrak{l} - \sudim
\mathfrak{l}_{\chi_n }).
\end{align*}
The theorem is now an easy consequence of Remark~\ref{rem:rea},
Remark~\ref{rem:semisimple}, Theorem~\ref{th:KW-N}, and
Theorem~~\ref{th:ME-simple}.
\end{proof}

\begin{corollary}
Assume that $\chi$ is regular semisimple (i.e. $\g_\chi$ is a
Cartan subalgebra $\h$ of $\g$). Then $U_\chi(\g)$ is a semisimple
superalgebra. Furthermore, the baby Verma modules $Z_\chi(\la)$
with $\la$ such that $\la (h)^p -\la(h) =\chi(h)^p$ for $h\in \h$
form a complete list of simple $U_\chi(\g)$-modules.
\end{corollary}

\begin{proof}
By Theorem~~\ref{th:morita}, $U_\chi(\g)$ is Morita equivalent to
$U_\chi(\h)$ which is semisimple by the assumption of regular
semisimplicity on $\chi$. Hence $U_\chi(\g)$ is a semisimple
superalgebra. Note that $\sudim \g - \sudim \g_{\chi} =2\dim \ev
{\mathfrak n} | 2 \dim \od {\mathfrak n}.$ By Theorem~\ref{th:KW},
the $\g$-module $Z_\chi(\la)$ is simple, since the dimension of
$Z_\chi(\la)$ is $p^{\dim \ev {\mathfrak n}} 2^{\dim \od
{\mathfrak n}}$. For different $\la$, $Z_\chi(\la)$ are
nonisomorphic by high weight consideration.
\end{proof}

\begin{corollary}
Let $\g=\mathfrak{osp}(1|2n)$ be the Lie superalgebra of type
B(0,$n$). Let $\chi$ be  regular nilpotent (i.e. the corresponding
$X \in \ev\g$ is a regular nilpotent element). Then, the baby
Verma modules of $\rea{\chi}{\g}$ are simple.
\end{corollary}

\begin{proof}
It is well known that $\dim \g_{X,\bar{0}} = n$, which is the rank
of $\ev\g=\mathfrak{sp}(2n)$. We claim that $\dim \g_{X,\bar{1}} =
1$. Indeed, $\od \g$ is the natural $\mathfrak{sp}(2n)$-module,
and $X$ can be regarded as a matrix of corank $1$. Therefore, we
have $\sudim \m'= \sudim \mathfrak{n}^-$ (recall the subalgebra
$\m'$ from Sect.~\ref{sec:subalgm}). Now having dimension equal to
$\dim \rea{\chi}{\m'}$, the baby Verma modules of $\rea{\chi}{\g}$
must be simple by Theorem~\ref{th:KW-N}.
\end{proof}

\section{The modular representations of $\osp(1 | 2)$}
\label{sec:osp12}

In this section, we give a complete description of the modular
representation theory of $\g =\osp(1|2)$, with many similarities
to the well-known $\mathfrak{sl}(2)$ case \cite{FP} (also cf.
\cite[Sect.~5]{Jan1}). It turns out that there are no projective
simple $U_\chi(\g)$-modules in contrast to the
$\mathfrak{sl}(2)$-case.

\subsection{Lie superalgebra $\osp(1|2)$}
Recall that $\g = \osp(1|2)$ consists of $3 \times 3$ matrices in
the following $(1|2)$-block form
\[
\begin{bmatrix}
0 & v & u \\
u & a & b \\
-v & c & -a
\end{bmatrix}
\]
with $a, b, c, u, v \in K$. The even subalgebra is generated by
\begin{gather*}
e =\begin{bmatrix} 0 & 0 & 0 \\ 0 & 0 & 1 \\ 0 & 0 & 0
\end{bmatrix} \qquad h=\begin{bmatrix} 0 & 0 & 0 \\ 0 & 1 & 0 \\ 0 & 0 &
-1 \end{bmatrix} \qquad f=\begin{bmatrix} 0 & 0 & 0 \\ 0 & 0 & 0
\\ 0 & 1 & 0 \end{bmatrix}.
\end{gather*}
A basis of $\od \g$ is given by
\begin{gather*}
E=\begin{bmatrix} 0 & 0 & 1 \\ 1 & 0 & 0 \\ 0 & 0 & 0
\end{bmatrix} \qquad F =\begin{bmatrix} 0 & 1 & 0 \\ 0 & 0 & 0 \\ -1 & 0 & 0
\end{bmatrix}.
\end{gather*}
The adjoint $\ev\g(\cong \mathfrak{sl}(2))$-module $\od\g$ is the
two-dimensional natural module.

We collect the commutation relations of these basis elements
below:
\begin{gather}
[h, E]=E \qquad [h, F] = -F  \notag\\
[e, E]=0 \qquad [e, F] = -E \notag\\
[f, E]=-F \qquad [f, F] =0 \notag\\
[E, E]= 2e \qquad [E, F] = h \qquad [F, F]= -2f. \notag
\end{gather}

It is easy to check that the relations $\pth e = 0 = \pth f$ and
$\pth h = h$ provide a restricted structure on the Lie
superalgebra $\g$.

Since $\ev \g \cong \mathfrak{sl}(2)$, there are three coadjoint
orbits of $\ev \g^*$ with the following representatives:
\begin{enumerate}
\item[(i)] regular nilpotent: $\chi(e) = \chi(h) =0$ and $\chi(f)
= 1$;

\item[(ii)] regular semisimple: $\chi(e)=0=\chi(f)$ and $\chi(h)=
a^p$ for some $a \in K^*$;

\item[(iii)] restricted: $\chi(e)=\chi(f)=\chi(h)=0$.
\end{enumerate}

\subsection{The baby Verma modules of $\osp(1|2)$}\label{subsec:osp-general}
Fix $\chi \in \ev \g^*$ such that $\chi(e)=0$ and denote the Borel
subalgebra $\mathfrak{b} :=Ke + Kh +K E$ with Cartan subalgebra
$\mathfrak{h}=K h$. Recall from Section~\ref{sec:ind-mod} the baby
Verma module $Z_{\chi}(\lambda) =U_\chi(\g)
\otimes_{U_\chi(\mathfrak b)} K_\la$ with $\lambda^p - \lambda =
\chi(h)^p$. The set $\{ v_i := F^i \otimes 1 \vert 0 \leq i <
2p\}$ is a basis for $Z_{\chi}(\lambda)$ with $\g$-action given by
\begin{align}
h v_i &= (\lambda - i) v_i .  \label{eq:osp(1|2)-1}\\
f v_i &= \begin{cases} -v_{i+2} & 0 \leq i < 2p -2,\\
\chi(f)^p v_0 & i = 2p - 2,\\
\chi(f)^p v_1 & i = 2p - 1.
\end{cases}  \label{eq:osp(1|2)-2}\\
e v_i &= \begin{cases} -\frac{i}{2} (\lambda + 1 -\frac{i}{2}
)v_{i-2} & \text{if $i$ is even,}\\
-\frac{i-1}{2} (\lambda  -\frac{i - 1}{2})v_{i-2} & \text{if $i$
is odd.}
\end{cases}  \label{eq:osp(1|2)-3}\\
F v_i &= \begin{cases} v_{i+1} & 0 \leq i < 2p -1,\\
-\chi(f)^p v_0 & i=2p-1.
\end{cases}  \label{eq:osp(1|2)-4}\\
E v_i &= \begin{cases} -\frac{i}{2} v_{i-1} & \text{if $i$ is
even,}\\
(\lambda - \frac{i - 1}{2}) v_{i - 1} & \text{if $i$ is odd.}
\end{cases}  \label{eq:osp(1|2)-5}
\end{align}
Here we use the convention that $v_i = 0$ when $i < 0$.

Let $M$ be an irreducible $\rea{\chi}{\g}$-module. Note that
$\chi(e) = 0$ and $E^2 = e$ so $E^{2p}= e^p = 0$. Thus the set $\{
m \in M \vert E\cdot m = 0 = e \cdot m\}$, which is $K h$-stable,
is nonzero. Hence there exists $0 \neq m_0 \in M$ such that $e
\cdot m_0 =0 = E \cdot m_0$ and $h \cdot m_0 = \lambda m_0$ for some
$\lambda \in K$, which satisfies $\lambda^p - \lambda =
\chi(h)^p$. Then by the Frobenius reciprocity there exists an
epimorphism of $\g$-modules:
$
Z_{\chi}(\lambda) \twoheadrightarrow M.
$
Thus by Remark~\ref{rem:rea}, every simple module appears as a
homomorphic image of some baby Verma module $Z_{\chi}(\lambda)$.

\subsection{The regular semisimple case}
In this case $\chi(h) \neq 0$, so $\lambda \notin \mathbb{F}_p$
for those satisfying $\lambda^p - \lambda = \chi(h)^p$, where
$\mathbb{F}_p$ denotes the finite field of $p$ elements. It
follows from (\ref{eq:osp(1|2)-1}-\ref{eq:osp(1|2)-5}) that the
only vectors annihilated by $E$ (and $e$) are scalar multiples of
$v_0$. Since each non-zero submodule (either graded or non-graded)
of $Z_{\chi}(\lambda)$ contains a nonzero vector killed by $E$,
$Z_{\chi}(\lambda)$ is irreducible and of type $M$. By
Section~\ref{subsec:osp-general} the baby Verma modules
$Z_{\chi}(\lambda)$ with $\lambda^p - \lambda = \chi(h)^p$ are all
the irreducibles and pairwise non-isomorphic. Now each
$Z_{\chi}(\lambda)$ is of dimension $2p$ and $\dim \rea{\chi}{\g}
= 4p^3$, the algebra $\rea{\chi}{\g}$ has to be isomorphic to the
semisimple superalgebra $\oplus_{\lambda} M_{2p}(K)$ by dimension
counting, where $M_d(K)$ denotes the simple algebra of all $d
\times d$ matrices over $K$. Summarizing, we have the following.

\begin{proposition} \label{prop:osp12semisimple}
Let $\g=\osp(1|2)$, and let $\chi \in \ev \g^*$ be regular
semisimple with $\chi(e) =\chi(f) =0$. Then
\begin{itemize}
\item[(i)] The algebra $\rea{\chi}{\g}$ is semisimple and
isomorphic to the algebra $\oplus_{\lambda} M_{2p}(K)$.

\item[(ii)] The algebra $\rea{\chi}{\g}$ has $p$ distinct
isomorphism classes of irreducible modules, each represented by
some $Z_{\chi}(\lambda)$ with $\lambda^p - \lambda = \chi(h)^p$.
The modules $Z_{\chi}(\lambda)$ are of type $M$.
\end{itemize}
\end{proposition}

\subsection{The regular nilpotent case}  In this case, $\chi(h) =
0$ and $\chi(f) = 1$. The $\lambda$ satisfying the equation
$\lambda^p - \lambda = 0$ lies in $\mathbb{F}_p = \{0, 1, \ldots,
p-1\}$. We observe that
$$
\{ v \in Z_{\chi}(\lambda) \vert E \cdot v =0\} = Kv_0 \oplus
Kv_{2\lambda + 1}.
$$
Assume $\lambda \neq \frac{p-1}{2}$ first. Note in this case $h$
acts on $v_0$ and $v_{2 \lambda + 1}$ with different eigenvalues.
Any (graded or non-graded) submodule of $Z_{\chi}(\lambda)$
contains a $Kh$-stable vector killed by $E$, hence either $v_0$ or
$v_{2 \lambda +1}$. But since $v_0 =- F^{2p - 2 \lambda -1}
v_{2\lambda +1}$, the submodule must be equal to
$Z_{\chi}(\lambda)$. Therefore $Z_{\chi}(\lambda)$ is irreducible
and of type $M$.

Now suppose $\lambda = \frac{p-1}{2}$, and so $2\la+1 =p$. In this
case, $v_0$ and $v_{2\la+1}$ have opposite $\Z_2$-parities but
identical $h$-eigenvaule. It follows that $Z_{\chi}(\lambda)$ is
irreducible of type $Q$. Note that $Z_{\chi}(\lambda)$ contains
two $p$-dimensional non-graded simple submodules
$Z_{\chi}(\frac{p-1}{2})^+ = K\{ v_i + \sqrt{-1} v_{p+i} \vert 0
\leq i <p\}$ and $Z_{\chi}(\frac{p-1}{2})^- = K\{ v_i - \sqrt{-1}
v_{p+i} \vert 0 \leq i <p\}$. There is an  odd $\g$-module
involution of $Z_{\chi}(\lambda)$ which exchanges $v_i$ and
$v_{i+p}$ for $0 \leq i \leq p-1$.

A second look at the space of vectors annihilated by $E$ shows
that $Z_{\chi}(\mu)$ is isomorphic to $Z_{\chi}(\lambda)$ if and
only if $\mu =\la$ or $\mu =\la^*$, where we denote $\lambda^* = p
- \lambda -1$. So a complete list of simple $\g$-modules consists
of $Z_{\chi}(\lambda)$, where $0\le \la \le\frac{p-1}{2}$.

>From now on, let $0\le \la,\mu \le\frac{p-1}{2}$. By the exactness
of the functor $\rea{\chi}{\g}
\otimes_{\rea{\chi}{\mathfrak{b}}}-$, the number of composition
factors isomorphic to $Z_{\chi}(\lambda)$ of the left regular
$\rea{\chi}{\g}$-module equals the number of composition factors
isomorphic to $K_{\lambda}$ or $K_{\lambda^*}$ in the left regular
$\rea{\chi}{\mathfrak{b}}$-module. This number is easily seen to
equal $2p$ if $\lambda = \frac{p-1}{2}$ (and so $\la^* =\la$) and
$4p$ otherwise.

Denote by $P_{\chi}(\lambda)$ the projective cover of
$Z_{\chi}(\lambda)$. We claim that the module $Z_{\chi}(\lambda)$
for each $\lambda$ is not projective. Otherwise,
$Z_{\chi}(\lambda)$ cannot appear as a section of any
$P_{\chi}(\mu)$ with $\mu \neq \la$ and it appears once as a
section of $P_{\chi}(\la) =Z_{\chi}(\lambda)$. For $\lambda \neq
\frac{p-1}{2}$, this would imply that the number of composition
factors isomorphic to $Z_{\chi}(\lambda)$ (which is of type $M$)
in $\rea{\chi}{\g}$ equals $\dim Z_{\chi}(\lambda)=2p$. This
contradicts $4p$ as claimed in case for $\lambda \neq
\frac{p-1}{2}$ in the preceding paragraph, and thus the module
$Z_{\chi}(\lambda)$ for $\lambda \neq \frac{p-1}{2}$ is not
projective. Now if $Z_{\chi}( \frac{p-1}{2})$ (which is of type
$Q$) were projective, then the number of composition factors
isomorphic to $Z_{\chi}(\frac{p-1}{2})$ in $\rea{\chi}{\g}$ equals
$\hf \dim Z_{\chi}(\frac{p-1}{2})=p$, by the Wedderburn theorem
for superalgebras (cf. \cite[Theorem~12.2.9]{Kle}), which
contradicts $2p$ as claimed in the preceding paragraph.

Note that $\rea{\chi}{\g}$ is a symmetric algebra by
Proposition~\ref{prop:Fro}. Then the head and socle of
$P_{\chi}(\lambda)$ for each $\la$ must be two distinct copies of
$Z_{\chi}(\lambda)$,
%
and hence $\dim P_{\chi}(\lambda) \geq 4p$. Now $\dim
P_{\chi}(\lambda) =4p$ for each $\la$ thanks to the following
calculation:
\begin{align}
4p^3 =\dim \rea{\chi}{\g} &= \sum_{0\le \lambda <\frac{p-1}{2}}2p
\cdot \dim P_{\chi}(\lambda) + p \cdot \dim
P_{\chi}(\frac{p-1}{2}) \notag\\
& \geq (p-1)/2 \cdot 2p \cdot 4p  + p\cdot 4p =4p^3. \notag
\end{align}

For $\lambda \neq \frac{p-1}{2}$, the endomorphism algebra of the
module $P_{\chi}(\lambda)$ is a two-dimensional local algebra with
basis $\{Id, \pi\}$, where $\pi$ is the projection of
$P_{\chi}(\lambda)$ to its socle, i.e.,
$$\text{End}_{\rea{\chi}{\g}}(P_{\chi}(\lambda)) \cong
K[x]/\langle x^2\rangle.$$

For $\lambda = \frac{p-1}{2}$, the endomorphism algebra
$\text{End}_{\rea{\chi}{\g}}(P_{\chi}(\frac{p-1}{2}))$ is
$(2|2)$-dimensional with basis $\{ Id, \pi \; (\text{even}); J',
\pi \circ J'\; (\text{odd}) \} $, where $\pi$ is the projection of
$P_{\chi}(\frac{p-1}{2})$ to its socle and $J'$ is the lift of the
odd automorphism of $Z_{\chi}(\frac{p-1}{2})$. We have the
following isomorphisms of algebras:
$$
\text{End}_{\rea{\chi}{\g}}(P_{\chi}(\frac{p-1}{2}))
 \cong K[x]/\langle x^2\rangle \otimes \mathfrak{q}_1(K)
\cong \mathfrak{q}_1(K[x]/\langle x^2\rangle).
$$

 Finally put $T = \oplus_{0 \leq \lambda< \frac{p-1}{2}}
P_{\chi}(\lambda)^{\oplus 2} \oplus P_{\chi}(\frac{p-1}{2})$. Then
the left regular module $\rea{\chi}{\g}$ is isomorphic to $T^p$
and
\begin{align}
\rea{\chi}{\g} &\cong
\text{End}_{\rea{\chi}{\g}}(\rea{\chi}{\g})^{\text op} \cong
\text{End}_{\rea{\chi}{\g}}(T^p)^{\text op} \cong
(M_p(\text{End}_{\rea{\chi}{\g}} (T)))^{\text op}   \notag\\
&\cong (\oplus_{0 \leq \lambda< \frac{p-1}{2}} M_{2p}(K[x]/\langle
x^2\rangle)
\oplus M_p(\mathfrak{q}_1(K[x]/\langle x^2\rangle)))^{\text op} \notag\\
&\cong (\oplus_{0 \leq \lambda< \frac{p-1}{2}} M_{2p}(K[x]/\langle
x^2\rangle) \oplus \mathfrak{q}_p(K[x]/\langle x^2\rangle))^{\text
op}. \notag
\end{align}
Summarizing, we have proved the following.

\begin{proposition}\label{prop:osp(1|2)-nil}
Let $\g=\osp(1|2)$, and let $\chi \in \ev \g^*$ be regular
nilpotent with $\chi(e) =\chi(h) =0$. Then
\begin{itemize}
\item[(i)] The superalgebra $\rea{\chi}{\g}$ has $\frac{p+1}{2}$
isomorphism classes of irreducible modules, i.e.
$Z_{\chi}(\lambda)$ for $\lambda = 0, 1, \ldots, \frac{p-1}{2}$.

\item[(ii)] For $\lambda \in \mathbb{F}_p$, the module
$Z_{\chi}(\lambda)$ is isomorphic to $Z_{\chi}(p - \lambda -1)$,
and there is no other isomorphism among the baby Verma modules.

\item[(iii)] The module $Z_{\chi}(\lambda)$ is of type $M$ for
$\lambda \neq \frac{p-1}{2}$; and $Z_{\chi}(\frac{p-1}{2})$ is of
type $Q$.

\item[(iv)] Each projective cover $P_{\chi}(\lambda)$
is a self-extension of $Z_{\chi}(\lambda)$.

\item[(v)]
As algebras, $\rea{\chi}{\g}^{\text op} \cong \oplus_{0 \leq
\lambda< \frac{p-1}{2}} M_{2p}(K[x]/\langle x^2\rangle) \oplus
\mathfrak{q}_p(K[x]/\langle x^2\rangle)$.
\end{itemize}
\end{proposition}

\subsection{The restricted case}
In this case $\chi =0$, and $\rea{\chi}{\g} = \rea{0}{\g}$ is the
restricted enveloping superalgebra. For $\lambda \in \mathbb{F}_p
=\{0, \ldots, p-1\}$, let $L(\lambda)$ be the $\g$-module with
basis ${v_0, \ldots, v_{2\lambda}}$ and with the action given by
formulas (\ref{eq:osp(1|2)-1}), (\ref{eq:osp(1|2)-3}),
(\ref{eq:osp(1|2)-5}) and
\begin{equation*}
f v_i = -v_{i+2}, \quad F v_i = v_{i+1}.
\end{equation*}
Each module $L(\lambda)$ is irreducible and of type $M$. Dropping
the subscript $\chi=0$, we shall denote by $Z(\la)$ and $P(\la)$
the baby Verma module and projective cover of $L(\la)$
respectively for each $\lambda \in \mathbb{F}_p$.

It is straightforward to verify that the baby Verma module
$Z(\lambda)$ has two composition factors $L(\lambda)$ and $L(p -
\lambda -1)$. Indeed, $Z(\lambda)$ has a unique proper submodule
generated by $v_{2\la +1}$, where $E.v_{2\la +1} =0$ by
(\ref{eq:osp(1|2)-5}). This submodule is simple and is isomorphic
to $L(p - \lambda -1)$.

The general results by Holmes and Nakano \cite{HN} apply in our
setup, since all the simple modules $L(\lambda)$ are of type $M$.
In particular, by \cite[Thms.~4.5 and 5.1]{HN} the projective
cover $P(\lambda)$ of $L(\lambda)$ has a baby Verma filtration,
and for any $\lambda , \mu \in \mathbb{F}_p$ one has the Brauer
type reciprocity
\[
( P(\lambda): Z(\mu) ) = [Z(\mu) : L(\lambda)],
\]
where $( P(\lambda): Z(\mu) )$ is the multiplicity of $Z(\mu)$
appearing in the baby Verma filtration of $P(\lambda)$, and
$[Z(\mu) : L(\lambda)]$ is the multiplicity of $L(\lambda)$ in a
composition series of $Z(\mu)$. It follows by the discussion on
the composition factors of $Z(\mu)$ that $(P(\lambda): Z(\mu)) =1$
for $\mu = \lambda \text{ or } p - \lambda -1$, and is $0$
otherwise.

In summary, we have proved the following.

\begin{proposition} \label{prop:res}
Let $\g=\osp(1|2)$. Then
\begin{itemize}
\item[(i)] The algebra $\rea{0}{\g}$ has $p$ isomorphism classes
of irreducible modules $L(\lambda)$ for $\lambda = 0, \ldots,
p-1$. Moreover, $L(\lambda)$ has dimension $2\lambda +1$.

\item[(ii)] For each $\la$, $Z(\la)$ has two composition factors:
$L(\lambda)$ and $L( p-\lambda-1)$.

\item[(iii)] For each $\la$, the projective cover $P(\lambda)$ has
a baby Verma filtration with $Z(p-\lambda-1)$ and $Z(\la)$ as
subquotients.
%
\end{itemize}
\end{proposition}

Clearly, Propositions~\ref{prop:osp12semisimple},
\ref{prop:osp(1|2)-nil}, and \ref{prop:res} fit well with the
Super KW Property established in Theorem~\ref{th:KW}.

\vspace{.3cm}

\noindent {\bf Acknowledgments.} It is a pleasure to thank Bin Shu
for discussions and help with modular Lie algebra literature in
the early stage of our formulation of the Super KW Conjecture in
2004. We are indebted to Sasha Premet for his influential ideas,
stimulating discussions and insightful suggestions. LZ thanks Dan
Nakano for his generous help with a reading course on modular Lie
algebras. WW also thanks Chaowen Zhang for sending a manuscript
during the preparation of this work which contained some basics on
modular Lie superalgebras and a (less precise) version of the
Super KW Conjecture as in Section~\ref{sec:basics}.


\begin{thebibliography}{ABCD}


\bibitem{BKN} B. Boe, J. Kujawa and D. Nakano,
{\em Cohomology and support varieties for Lie superalgebras II},
Proc. London Math. Soc., to appear, arXiv:0708.3191, 2007.

\bibitem{Ca} R. Carter, {\em Finite groups of Lie type: Conjugacy classes and complex
characters}, Pure and Applied Math.,
John Wiley \& Sons, Inc., New York, 1985.

\bibitem{Cur} C. Curtis,
{\em Noncommutative extensions of Hilbert rings}, Proc. Amer.
Math. Soc. {\bf 4} (1953), 945--955.

\bibitem{Far} R. Farnsteiner,
{\em Note on Frobenius extensions and restricted Lie
superalgebras}, J. Pure Appl. Algebra, {\bf 108} (1996), 241--256.

\bibitem{FP1} E. Friedlander and B. Parshall, {\em
Support varieties for restricted Lie algebras}, Invent. Math. {\bf
86} (1986), 553--562.

\bibitem{FP} E. Friedlander and B. Parshall, {\em Modular
representation theory of Lie algebras}, Amer. J. Math. {\bf 110}
(1988), 1055--1093.

\bibitem{HN} R. Holmes and D. Nakano,
{\em Brauer-type reciprocity for a class of graded associative
algebras}, J. Algebra {\bf 144} (1991), 117--126.


\bibitem{Jan1} J. Jantzen, {\em Representations of Lie algebras in prime
characteristic}, NATO Adv. Sci. Inst. Ser. C Math. Phys. Sci.,
{\bf 514}, Representation theories and algebraic geometry
(Montreal, PQ, 1997), 185--235, Kluwer Acad. Publ., Dordrecht,
1998.

\bibitem{Jan2} J. Jantzen, {\em Nilpotent orbits in representation
theory}, 1--211, Progr. Math., {\bf 228}, Birkh\"{a}user Boston,
Boston, MA, 2004.


\bibitem{Kac} V. Kac, {\em Lie Superalgebras}, Adv. Math. {\bf 16}
(1977), 8--96.

\bibitem{Kle} A. Kleshchev,
{\em Linear and projective representations of symmetric groups},
Cambridge Tracts in Mathematics {\bf 163}, (Cambridge University
Press, Cambridge, 2005)h.


\bibitem{Pr1} A. Premet,
{\em Irreducible representations of Lie algebras of reductive
groups and the Kac-Weisfeiler conjecture}, Invent. Math. {\bf 121}
(1995), 79--117.


\bibitem{Pr2}
A. Premet, {\em Special transverse slices and their enveloping
algebras}, Adv. Math. {\bf 170} (2002), 1--55.


\bibitem{SNR1}
M. Scheunert, W. Nahm, and V. Rittenberg, {\em Classification of
all simple graded Lie algebras whose Lie algebra is reductive I},
J. Math. Phys. {\bf 17} (1976), 1626--1639.

\bibitem{SNR2}
M. Scheunert, W. Nahm, and V. Rittenberg, {\em Classification of all
simple graded Lie algebras whose Lie algebra is reductive II.
Construction of the exceptional algebras}, J. Math. Phys. {\bf 17}
(1976), 1640--1644.


\bibitem{Ser} V. Serganova,
{\em Kac-Moody superalgebras and integrability}, Preprint 2007,
available in http://math.berkeley.edu/~serganov

\bibitem{SW} B. Shu and W. Wang,
{\em Modular representations of the ortho-symplectic supergroups},
Proc. London Math. Soc. {\bf 96} (2008), 251--271.


\bibitem{Skr} S. Skryabin,
{\em Representations of the Poisson algebra in prime
characteristic}, Math. Z. {\bf 243} (2003), 563--597.

\bibitem{W} C.T.C. Wall,
{\em Graded Brauer groups}, J. Reine Angew. Math. {\bf 213} (1964), 187--199.

\bibitem{KW} B. Weisfeiler and V.~Kac,
{\em On irreducible representations of Lie $p$-algebras}, Func.
Anal. Appl. {\bf 5} (1971), 111--117.


\end{thebibliography}
\end{document}